\documentclass[12pt,reqno]{article}
\usepackage[utf8]{inputenc}
\usepackage{enumerate}
\usepackage{lmodern}
\usepackage[T1]{fontenc}
\usepackage{verbatim}
\usepackage{textcomp}

\usepackage[english]{babel}
\usepackage[a4paper,vmargin={3.5cm,3.5cm},hmargin={2.5cm,2.5cm}]{geometry}
\usepackage[font=sf, labelfont={sf,bf}, margin=1cm]{caption}
\usepackage[pdftex]{hyperref}
\usepackage[pdftex]{color,graphicx}
\usepackage{xcolor}

\usepackage{amsmath,amsfonts,amssymb,amsthm,mathrsfs,mathtools,bbm}
\usepackage{empheq}
\newtheorem{theorem}{Theorem}[section]
\newtheorem{corollary}[theorem]{Corollary}
\newtheorem{lemma}[theorem]{Lemma}
\newtheorem{proposition}[theorem]{Proposition}

\theoremstyle{remark}

\theoremstyle{definition}

\def\RR{\mathbb{R}}
\def\R{\mathbb{R}}
\def\ZZ{\mathbb{Z}}
\def\NN{\mathbb{N}}
\def\QQ{\mathbb{Q}}
\newcommand{\PP}{\mathbb{P}}
\renewcommand{\Pr}{\mathbb{P}}
\def\XX{\mathbb{X}}

\def\EE{\mathbb{E}}
\newcommand{\E}{\mathbb{E}}
\newcommand{\xxi}{\zeta}

\newcommand{\II}{{S}}

\def\1{{\bf 1}}
\def\al{\alpha}

\def\ep{\varepsilon}

\def\la{\lambda}

\def\bx{\mathbf{x}}
\def\by{\mathbf{y}}

\def\bz{\mathbf{z}}

\def\b0{\mathbf{0}}

\newcommand{\bp}{p}
\newcommand{\bq}{q}

\newcommand{\bfe}{{e}}

\def\d2{d_2}

\def\scr{\mathscr}

\newcommand{\tg}{\tilde{g}}

\newcommand{\bea}{\begin{eqnarray}}
\newcommand{\eea}{\end{eqnarray}}
\newcommand{\bean}{\begin{eqnarray*}}
\newcommand{\eean}{\end{eqnarray*}}
\newcommand{\eps}{\varepsilon}

\newcommand{\cI}{{\cal I}}
\renewcommand{\emptyset}{\varnothing}

\numberwithin{equation}{section}

\newcommand{\blue}[1]{{\color{black}{#1}}}

\begin{document}
\title{\bf On the
rate of convergence in 
the Hall-Janson coverage theorem 
}
\author{ Mathew D. Penrose$^{1,2}$ and Xiaochuan Yang$^{1,3}$  \\
{\normalsize{\em University of Bath and Brunel University}} }

 \footnotetext[1]{ Supported by EPSRC grant EP/T028653/1 }

\date{\today}

\maketitle

\begin{abstract}   
	Consider a spherical Poisson Boolean model $Z$ in Euclidean $d$-space with $d \geq 2$, with Poisson intensity $t$ and radii distributed like $rY$ with $r \geq 0$ a scaling parameter and $Y$ a fixed nonnegative random variable with
	finite $(2d-2)$-nd moment (or if $d=2$, a finite
	$(2+\eps)$-moment for some $\eps >0$).
	Let $A \subset \R^d$ be compact with a nice boundary.  Let $\alpha $ be the expected volume of a ball of radius $Y$, and suppose $r=r(t)$ is chosen so that $\alpha t r^d - \log t - (d-1) \log \log t$ is a constant independent of $t$.  A classical result of Hall and of  Janson determines the (non-trivial) large-$t$ limit of the  probability that $A$ is fully covered by $Z$. In this paper we provide an $O((\log \log t)/\log t)$ bound on the rate of convergence in that result. With a slight adjustment to $r(t)$, this can be improved to $O(1/\log t)$.

	\footnotetext[2]{Corresponding author: Department of Mathematical Sciences, University of Bath,
 Bath BA2 7AY, UK: \texttt{m.d.penrose@bath.ac.uk} }
	\footnotetext[3]{\texttt{xiaochuan.j.yang@gmail.com} }

\end{abstract}

\section{Introduction and statement of result}
\label{secintro}

A classical problem in stochastic geometry 
\cite{Aldous,HallBk} 
is to determine
the probability of fully covering a region of Euclidean
$d$-space 
with a union of deterministic or random shapes centred on random points.
There are applications in topological data analysis, wireless communications,
and also in biology and chemistry, as discussed in e.g. \cite{HallZW,HallBk,Janson,P23,PH24}.
An asymptotic
answer to this question was provided in the 1980s by 
Hall \cite{HallZW} and Janson \cite{Janson} (Theorem \ref{t:H-J} below).
In this paper we give a more quantitative version of that result.

Let $d \in \NN$ and
let $\QQ$ be a probability measure on $\R_+ := [0,\infty)$.
Given $t >0$, let $\xi_t$ denote a Poisson process in $\R^d \times \R_+$
with intensity measure $t \lambda_d \otimes \QQ$, where
 $\la_d$ denotes the Lebesgue measure on $\RR^d$.
Given also $r >0$, and $k \in \NN$,
define the {\em $k$-occupied set} $Z_k(\xi_t,r)$ and
{\em $k$-vacant set} $V_k(\xi_t,r)$ where for
 any simple counting measure $\mu$ on $\RR^d\times \R_+$ and $r\geq 0$,
 we define  
$$
Z_k(\mu,r):= \{y \in \R^d: \# \{(x,a) \in \mu: y \in B(x,ra)\} \geq k\};
~~~~~~ V_k(\mu,r)= \R^d \setminus Z_k(\mu,r), 
$$
where for $x \in \R^d$ and $s >0$,
we set $B(x,s):= \{y \in \R^d: \|y-x\|\leq s\}$, and
$\|\cdot\|$ is the Euclidean norm.
 Then $Z_k(\xi_t,r)$ is the set of locations that are covered
 at least $k$ times by
a {\em spherical Poisson Boolean model}
(SPBM)
consisting of balls (known as {\em grains})
with independent $\QQ_r$-distributed random radii centred
on the points of a homogeneous Poisson process of intensity $t$ in $\R^d$.
Here $\QQ_r$ denotes the distribution of $rY$, and
$Y$ denotes a random variable with distribution $\QQ$.

%
%

Let $A$ denote a compact subset of $\RR^d$, with interior denoted
$A^o$. Assume throughout 
	that $ \lambda_d(A)= \lambda_d(A^o) >0$.
This paper is concerned with the asymptotic as $t\to\infty$ of 
the {\em coverage threshold} $R_{t,k}$ defined by 
\begin{align*}
	R_{t,k} := \inf\{r>0: A \subset Z_k(\xi_t,r)\}.
\end{align*}
%
Let
   $\theta_d :=\la_d(B_1(o)) = \pi^{d/2}/\Gamma (1+d/2)$, the
   volume of the $d$-dimensional unit ball.
   Here $o$ denotes the origin in $\R^d$.
 The following result identifies the weak limit of $R_{t,k}$.
 
\begin{theorem}[Hall, Janson]
	\label{t:H-J}
	Suppose $0 < \EE[Y^{d+\eps}] < \infty$ for some $\eps >0$.
	Set $\al =  \theta_d \EE[Y^d]$.
	Let $k\in\NN$,  
	$\beta \in\RR$ and define $(r_t)_{t >1}$ taking values
	in $\R_+$ by the relation  
\begin{align}\label{e:r_t0}
	\al t r_t^d = \left( \log t + ( d+k-2 )
	\log\log t + \beta \right) \vee 0,
\end{align}
	where we use $u \vee v$ for the maximum of two real numbers
	$u,v$.
	Then  
\begin{align*}
	\lim_{t\to\infty} \PP[R_{t,k}\le r_t] = \exp( - c_{d,k,Y}
	\la_d(A) e^{-\beta} )
\end{align*}
	where the constant $c_{d,k,Y}$ is given by 
	\begin{align}
		c_{d,k,Y} := \frac{1}{d!(k-1)!}
		\Big(\frac{\sqrt{\pi} \Gamma(1 + d/2)}{\Gamma((d+1)/2)}\Big)^{d-1} \frac{(\EE[Y^{d-1}])^d}{(\EE[Y^{d}])^{d-1}}.
		\label{e:cdkY}
	\end{align}
\end{theorem}
The case $k=1$ of
Theorem  \ref{t:H-J} is immediate from \cite[Theorem 2]{HallZW}. The
general case
can be derived from
the result in \cite{Janson}; see
\cite[Lemma 7.2]{P23} and \cite[Lemma 3.2]{PH24}.


Another way to formulate this theorem is the convergence in distribution
$$
\al t R_{t,k}^d - \log t -  (d+k-2)\log\log t \to G+
\log (c_{d,k,Y}\lambda_d(A)),
$$
where $G$ is a standard Gumbel random variable.

Previous proofs of Theorem \ref{t:H-J} do not
provide any quantitative error bound for finite $t$, or any rate of convergence.
Our aim in this paper is to provide a new proof
(or at least, one with some new ingredients)
that gives a rate of convergence in the Hall-Janson
theorem.  We do so under the extra assumption that $d \geq 2$;
when $d \geq 3$ we also require
a stronger moment condition on $Y$ than in Theorem \ref{t:H-J}.


If $f(t) \in \R$ and $g(t) \in (0,\infty)$ are defined 
for all large enough $t \in \R$, we say $f(t) = O(g(t))$ as
$t \to \infty $ if
$\limsup_{t \to \infty} (|f(t)|/g(t)) < \infty$.

\begin{theorem}  
	\label{t:main}
	Suppose $0 < \E[Y^{2d-2}] < \infty$ and
	 $0 < \E[Y^{d+\eps}] < \infty$ for some $\eps >0$.
	  Assume that $d \geq 2$,
	let $k \in \NN$,
	and let $(r_t)_{t>1}$ taking values in $\R_+$ be given
	by the relation
\begin{align}\label{e:r_t}
	\al t r_t^d = \left( \log t + ( d+k-2 ) \log\log t + \beta 
	+ \frac{(d+k-2)^2 \log \log t}{\log t}
	\right) \vee 0.
\end{align}
	As $t\to\infty$, it holds that
	\begin{align}
		\PP[R_{t,k}\le r_t] - \exp( - c_{d,k,Y} \la_d(A) e^{-\beta})
		= 
		O \Big(\la_d(A' \setminus A'') + \frac{1}{\log t}\Big),
		\label{e:JQ}
	\end{align}
		where $A' := A'(t) := A\oplus B(o,\sqrt{r_t})$ and
		$\oplus $ denotes the Minkowski sum,
		while $A'' := A''(t):= \{x \in A: B(x,\sqrt{dr_t}) \subset A\}$.
\end{theorem}

We note that if $A$ has a smooth or polytopal boundary, then the
right hand side of \eqref{e:JQ} is simply $O(1/(\log t))$.

We sketch the main ideas of our proof. Let us refer to the closure of
each connected component of $V_k(\xi_t,r_t)$ as a {\em vacant region}.
A percolation argument shows that with high probability, all vacant regions
intersecting $A$ are confined to the slightly bigger region $A'$.
The lowest point of each vacant region is a
local minimum of the closure of $V_k(\xi_t,r_t)$, and the mean number
of local minima in $A$ or in $A'$
can be computed using the multivariate Mecke formula.
We divide $A$ into moderately small boxes and  
show by a second moment calculation (again using the Mecke formula)
that the probability of there being a box containing
a pair of local minima  
decays like $1/\log t$ (it is this step in the proof that takes the most work).
One can then use independence of
contributions of different boxes, and standard Poisson-binomial
approximation, to get the result. 

For $k=1$ one can \blue{alternatively} use a version
of the Chen-Stein method 
to show that the number of such local minima 
in $A$ is approximately Poisson;
see the earlier version  of this paper \cite{PY24}.
Unlike that method,
the box method here works for general $k$.

If we were to define $(r_t)_{t >1}$ by the simpler relation 
\eqref{e:r_t0} rather than by \eqref{e:r_t}, we would
get \eqref{e:JQ} with
$O((\log \log t)/\log t)$ instead of
$O(1/\log t)$ on the right. This can be done  by following 
the proof  of Theorem \ref{t:main} with a few modifications,
mainly in the 
proof of Proposition \ref{p:meanlim}.

We now identify various possible extensions to our result that
would be of interest.

We restrict attention to grains that are Euclidean balls of possibly 
random radius.  More general distributions for the shapes of the
grains in the Boolean model 
are considered both by Hall  \cite{HallZW} and by Janson \cite{Janson};
for example, a uniformly randomly rotated random polygon.
It would be natural to try to extend our result to these settings too.

Hall \cite{HallZW} shows that for $k=1$, the number 
of vacant regions intersecting $A$ is asymptotically Poisson
with parameter  $c_{d,1,Y} \lambda_d(A) e^{-\beta}$. 
It may be possible  to extend our method
to give a bound on the rate of convergence in total variation distance.
Compared to the argument given here, the main missing step would be
to show that with high probability each vacant region within $A'$
has exactly one
local minimum (for our result here, all we need is the fact that each
vacant region has at least one local minimum).

Our method, in particular, the proof of Proposition \ref{p:case2},
does not appear to work for $d=1$.

In \cite{P23} and \cite{PH24}, analogous results to Theorem \ref{t:H-J}
are presented for the {\em restricted} spherical Poisson Boolean model where
one discards all grains centred outside $A$, and considers 
coverage of $A$ by
the remaining grains.  The proof of these results relies heavily on
Theorem \ref{t:H-J}. Using Theorem \ref{t:main}, it may
be  possible to obtain rates of convergence
in those results too, at least for $d \geq 3$.
 (The coverage process on the boundary of $A$ is
a $(d-1)$-dimensional SPBM, so we may need
a rate of convergence for the unrestricted SPBM in $d=1$
to derive a rate for the restricted SPBM in  $d=2$.)

Let $B \subset \R^d$  be compact with $A\subset B^o$.
Janson \cite{Janson} proves a variant of Theorem \ref{t:H-J} where
the Poisson point process $\xi_t$ is replaced by 
a {\em binomial} point process in $B \otimes \R_+$, given by a collection of 
$n$ random elements of $B \times \R_+$ with probability
distribution $\lambda_d(B)^{-1} \lambda_d|_B \otimes \QQ$, with $n \in \NN$,
as $n \to \infty$ with the scaling parameter $r$ now depending on $n$.
It would be of interest to provide a bound on the
rate of convergence for this result too.

It follows from our result that if $A$ has a nice boundary
then there exist positive finite  constants $c', t_0 $ such that
$|\PP[R_{t,k} \leq r_t] - \exp(-c_{d,k,Y} \lambda_d(A) e^{-\beta})|$
is bounded by $c'/\log t$ for all  $t >t_0$. In principle, one
could obtain numerical values for $c'$ and $t_0$ by going through
our proof. For $d=2$, $Y \equiv 1$,
numerical upper and lower bounds 
for the probability of coverage can be found in \cite{HallBk, Lan}.
Those bounds are not asymptotically  sharp.


\section{Reduction to a counting problem}
We now embark on proving Theorem \ref{t:main}.
We fix $d$, $k$,
$\QQ$, $A$ and $\beta$, and assume from now on that $(r_t)_{t > 1}$
is given by \eqref{e:r_t}.

We can and do assume without loss of generality that $
\QQ(\{0\}) = \PP[Y=0]=0$. Indeed,
the coverage capability of 
$Z_k(\xi_t,r)$ in the case $\PP[Y=0]>0$ is the same
as that of  $Z_k(\eta_t, r$) where $\eta_t$ is a Poisson process
with intensity $\PP[Y>0] t \lambda_d \otimes \QQ[Y\in dy | Y>0]$.

For any $x, x'\in \RR^d$, we say $x$ is {\em lower}
than $x'$ if 
$\langle x, e_d \rangle < \langle x', e_d \rangle$,
where $e_d:= (0,\ldots,0,1)$ is the $d$-th coordinate \blue{unit} vector and
$\langle \cdot, \cdot \rangle$ is the Euclidean inner product.
%
Let $x_1,\ldots,x_d$ be distinct vectors in $\RR^d$,
$r>0$, and $a_1,...,a_d\in \R_+$. 
For $i \in [d]:= \{1,\ldots,d\}$,
set $\bx_i = (x_i, a_i)$. 
If  $\cap_{i\in[d]}\partial B(x_i, r a_i)$
contains exactly two points, write these as  
$p = p^{(r)}(\{\bx_1,\ldots,\bx_d\}),q = q^{(r)} 
(\{\bx_1,\ldots,\bx_d\})$
with $p$ lower than $q$, or  with $p$ preceding $q$ lexicographically
if
$\langle p,e_d \rangle = \langle q, e_d \rangle$.  
(Here $\partial D := D \setminus D^o$ for any compact $D \subset \R^d$.)
	Define 
	the indicator function 
\begin{align}
	h^{(r)}(\{\bx_1,...,\bx_d\})  := 
	{\bf 1}\{
		&
	\#(\cap_{i\in[d]}\partial B(x_i, r a_i)) =2
	\}
	\times {\bf 1} \{q^{(r)}(\{\bx_1,\ldots, \bx_d\})
	\mbox{ is a}
\nonumber	\\
	&
	\mbox{ local minimum of ~}
		\R^d \setminus 
	\cup_{i\in[d]} B(x_i, r a_i)^o
	\},
	\label{e:defh}
\end{align}
where we say a point $x$   is a {\em local minimum} 
of a set $D \subset \R^d$ if $x \in D$, and for some neighbourhood
$U$ of $x$, $x$ is lower than any other point of $D \cap U$.
Also for any $\psi \subset \R^d \times \R_+$ with $\#(\psi) \neq d$,
set $h^{(r)}(\psi) =0$. Three examples are shown in Figure
\ref{f:locmin}.

        \begin{figure}[!h]

\center
\includegraphics[width=10cm]{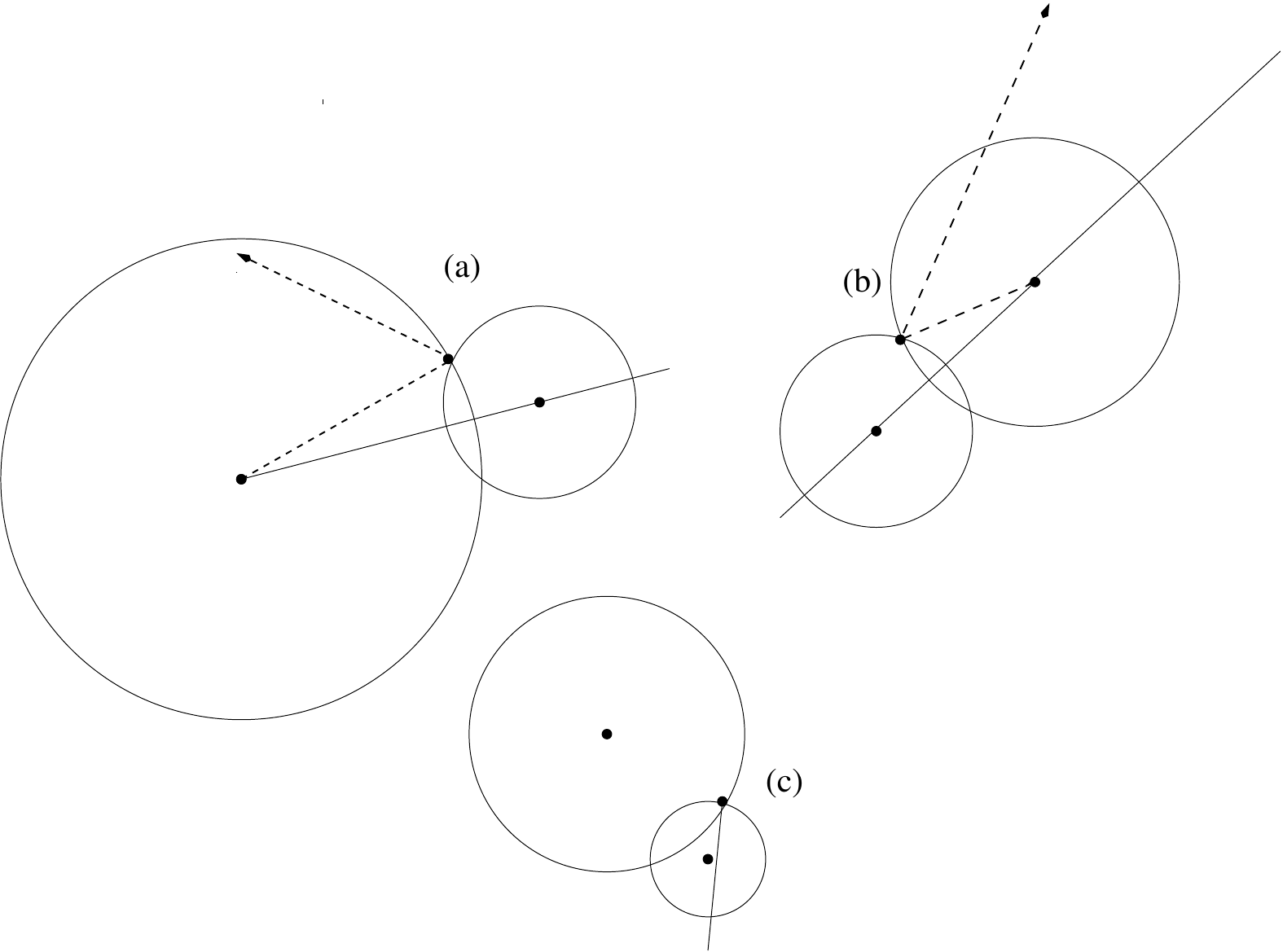}
                \caption{\label{f:locmin}Three examples in $d=2$
		of a pair $\{x_1,x_2\}$ (the centres of two
		overlapping circles) and associated
		$q^{(r)}(\{\bx_1,\bx_2\})$ 
		(the dot shown on the
		 intersection of circles). In (a) and (b)  
		$q^{(r)}(\{\bx_1,\bx_2\})$ is a local minimum
		of the complementary region, but in (c) it is not. 
        }
\end{figure}

For any Borel-measurable set $ D \subset \RR^d$, $r > 0$,
and $\psi \subset \R^d \times \R_+$, define
$$
h^{(r)}_D (\psi) := 
h^{(r)}(\psi) \1 \{q^{(r)}(\psi)\in D\}.
$$
Let $\xxi >0$ be a small fixed constant satisfying
\begin{align}
	0 < \xxi < 
	\frac{\EE[\min(Y,1/2)^d]}{8 d \EE[Y^d]} . 
	\label{e:xxi}
\end{align}
We decompose $\xi_t$ as $\xi'_t \cup \xi'' (t)$, where we
set
$$
\xi'_t : \xi \cap (\R^d \times [0,t^\xxi]); 
~~~~~~~~~~~~~~
\xi''_t : \xi \cap (\R^d \times (t^\xxi,\infty)).
$$

Define the random variable
\begin{align}
	F_{t}(D) :=
	\sum_{\psi \subset  \xi'_t}
	h^{(r_t)}_D (\psi)\1
\{ \bq^{(r_t)}(\psi) \in V_k(\xi'_t \setminus 
	\psi , r_t)  
\}.
	\label{e:Ftdef}
\end{align}
In this section we show that we can approximate $\PP[R_{t,k} \leq r_t]$
from above 
with  $\PP[F_t(A^o)=0]$, and from below with $\PP[F_t(A')=0]$, up to
small correction.

Before starting on this,
we give two alternative characterisations of the `local minimum'
condition in the definition of $h^{(r)}(\cdot)$ at \eqref{e:defh}.

Given distinct $q,x_1,\ldots,x_d \in \R^d$, we define the indicator
function $g(q,\{x_1,\ldots,x_d\} )$ to be 1, if and only if the following
composite
{\em hyperplane condition} holds:

	(i)
for any hyperplane $H$ that contains $q$ but
not $q + \bfe_d$
at least one of $x_1, \ldots, x_d$ is lower than $H$ \blue{(i.e.,
strictly on the opposite side of $H$ to $q+e_d$)}; and

(ii) there is no hyperplane $H$ containing both $q$ and $q + e_d$ 
such that all
points of $\{x_1,\ldots,x_d\} \setminus H$ lie on the same side of $H$.

Moreover, we define the indicator function $\tg(q,\{x_1,\ldots,x_d\})$ 
to be 1 if and only if $e_d \in  {\rm Cone} (q-x_1,\ldots,q-x_d)^o$,
the interior of the conical hull of $\{q-x_1,\ldots,q-x_d\}$.

\begin{lemma}
	\label{l:hgg}
	Let $r,a_1,\ldots,a_d \in (0,\infty)$. Then for
	all $(x_1,\ldots,x_d) \in (\RR^d)^d$
	such that $\#(\cap_{i=1}^d \partial B(x_i,ra_i)) =2$,
	 setting $\bx_i := (x_i,a_i)$ for each $i \in [d]$,
	 we have that
	\begin{align}
		h^{(r)}(\{\bx_1,\ldots,\bx_d\}) 
		& = g(q^{(r)} 
		(\{\bx_1,\ldots, \bx_d\}),\{x_1,\ldots,x_d\})
		\nonumber
		\\
		& = \tg(q^{(r)}(\{\bx_1,\ldots, \bx_d\}),\{x_1,\ldots,x_d\}).
		\label{e:hgg}
	\end{align}
\end{lemma}
\blue{
Lemma \ref{l:hgg} is illustrated in Figure \ref{f:locmin}. There,
the lines between circle centres show that $g(q^{(r)}(\{\bx_1,\bx_2\}),
\{x_1,x_2\}) =1$ (i.e., the
hyperplane condition holds) for (a) and (b),
and the  line on (c) shows that the hyperplane condition fails in that case
since that line passes through $q^{(r)}(\bx_1,\bx_2)$ and lies below
both $x_1$ and $x_2$. The dotted lines show that in (a) and (b), an upward vertical vector
can be made as a positive linear combination of $q^{(r)}(\bx_1,\bx_2)-x_1$
and $q^{(r)}(\bx_1,\bx_2) - x_2$, showing $\tg(q^{(r)}(\{\bx_1,\bx_2\}),
\{x_1,x_2\})=1$
in those cases (but not for (c)).}
\begin{proof}[\blue{Proof of Lemma \ref{l:hgg}}]
Suppose  $x_1,\ldots,x_d \in \R^d$ with
	$\#\cap_{i=1}^d \partial B(x_i,ra_i) =2$ and
	set $q := q^{(r)}(\{\bx_1,\ldots,\bx_d\})$.
	We show first that $h^{(r)}(\{\bx_1,\ldots,\bx_d\})
	\leq g(q,\{x_1,\ldots,x_d\})$.
	Suppose that $g(q,\{x_1,\ldots,x_d\}) = 0$.  Then either 
	(i) or (ii) in the definition of the hyperplane condition
	fails. We claim that in either case $h^{(r)}(\{\bx_1, \ldots,\bx_d\})
	=0$.

	Suppose there is a hyperplane $H$ containing $\bq$ but not $\bq+ \bfe_d$
	with no $x_i$'s below $H$. 
	Let $e $ be a unit vector orthogonal to $H$ with $\langle e, e_d 
	\rangle < 0$.
	Then for all $b >0$ and all $i \in [d]$, $q+ be$ lies below
	$q$ and hence not on the same side of $H$ as $x_i$, so
	that $\| q+ be - x_i\| > \|q - x_i\| = ra_i$. Hence
	 $\bq + b e$ lies
	in $\R^d \setminus \cup_{i=1}^d B(x_i, ra_i)$
	and below $\bq$, and therefore
	$q$ is not a local minimum of 
	$\RR^d \setminus \cup_{i=1}^d B(x_i,ra_i)^o$
	so $h^{(r)}(\{\bx_1,\ldots,\bx_d\}) =0$.

	Suppose instead that
	there exists a hyperplane $H$ with $q,q + e_d \in H$
	and all points of $\{x_1,\ldots,x_d\} \setminus H$ on the same
	side of $H$. Then taking $e$ orthogonal to $H$ and away
	from all $x_i$, we have for all  $b >0$ that
	$q + b e \in \R^d \setminus \cup_{i=1}^d B(x_i,ra_i)$
	and hence $h^{(r)}(\{\bx_1,\ldots,\bx_d\}) =0$. Thus our earlier
	claim is justified, and thus
	$h^{(r)}(\{\bx_1,\ldots,\bx_d\})
	\leq g(q,\{x_1,\ldots,x_d\})$.

	Conversely, if $h^{(r)}(\{\bx_1,\ldots,\bx_d\}) =0$ then
	 $q $ is not a local minimum
of $ \R^d \setminus \cup_{i\in[d]} B(x_i, r a_i)^o
	$, so
	there is a sequence $(z_n)_{n \in \NN}$ taking values
	in $\RR^d \setminus \{o\}$
	with
	$z_n \to o$ as $n \to \infty$, such that
	$\langle z_n,e_d \rangle \leq 0$ 
	and $q + z_n \notin \cup_{i=1}^d B(x_i,ra_i)$.
	Taking $u_n := \|z_n\|^{-1}z_n$ and taking a subsequence
	if needed, we have $u_n \to u$ for some $u \in \mathbb{S}_{d-1}
	:=\partial  B(o,1)$ 
	with $\langle u,e_d \rangle \leq 0$.
	Also we must have $\langle u,q-x_i \rangle
	\geq 0$
	for each $i \in [d]$, 
	since 
	otherwise there exists $i$ such that for large $n$ we have
	$q+z_n\in B(x_i,r a_i)$. Hence, 
	writing $H$ for the hyperplane through $q$
	perpendicular to $u$,  if $q+ e_d \notin H$ we must have
	that no $x_i$ lies below $H$.
	If $q+e_d \in H$ then none of the $x_i$ lies strictly on the same side
	of $H$ as $q+u$.
	Thus the hyperplane condition fails and $g(q,\{x_1,\ldots,x_d\}) =0$.
	Thus $g(q,\{x_1,\ldots,x_d\}) \leq h^{(r)}(\{\bx_1,\ldots,\bx_d\})$,
	and combined the earlier claim this gives us the first
	equality of \eqref{e:hgg}.

	Next we show that $\tg(q,\{x_1,\ldots,x_d\}) \leq 
	g(q,\{x_1,\ldots,x_d\})$.
	For $i \in [d]$ 
	set $u_i : = q- x_i$.
	Suppose $e_d \in {\rm Cone}(u_1, \ldots, u_d)^o$.
	Write $e_d = \sum_{i=1}^d b_i u_i$ with $b_i > 0$ for all $i
	\in [d]$.
	Then for
	any hyperplane $H$ containing $q$ but not $q+ e_d$, setting $f$
	to be a downward unit vector orthogonal to $H$
	we have 
	$$
	0 > \langle f, e_d \rangle = \sum_{i=1}^d b_i \langle f, u_i \rangle, 
	$$
	so for at least one $i$ we have $\langle f, u_i \rangle  < 0$
	and hence $x_i $ lies below $H$.

	If there is a hyperplane $H$ containing $q$ and $q+e_d$ with
	all points of $\{x_1,\ldots,x_d\} \setminus H$ on the same
	side of $H$, let $f$ be orthogonal to $H$ with $\langle f, u_i\rangle
	 \leq 0$ for all $i \in [d]$. Then $0= \langle f, e_d \rangle
	= \sum_{i=1}^d b_i \langle f, u_i \rangle$, so
	$\langle f, u_i \rangle =0 $ for all $i$. This would imply that
	$x_1,\ldots,x_d$ all lie in $H$ and $\cap_{i=1}^d \partial
	B(x_i,r a_i) = 
	\{q\}$,  contrary to our assumption that $\#(\cap_{i=1}^d \partial
	B(x_i,ra_i)) = 2$.

	Thus the hyperplane condition holds so $g(q,\{x_1,\ldots,x_d\}) =1$.
	Therefore we have
	$\tg(q,\{x_1,\ldots,x_d\}) \leq g(q,\{x_1,\ldots,x_d\})$ as claimed. 

	To prove the reverse inequality, suppose 
	$\tg(q,\{x_1,\ldots,x_d\}) =0 $.

	If $e_d \notin {\rm Cone}(u_1,\ldots,u_d)$,
	  let $z$ be the nearest point to 
	$e_d $ of  ${\rm Cone}(u_1,\ldots,u_d)$,
	and let $H$ be the hyperplane through $z$ orthogonal to
	$z-e_d$. Then $o \in H$ and all of ${\rm Cone}(u_1,\ldots, u_d) 
	\setminus
	H $
	lies below $H$. Let $H':= H +q := H \oplus \{q\}$.
	Then for each $i \in [d]$,
	$q+ u_i $ is on or below $H'$ so $x_i = q-u_i$
	is not lower than $H'$. 
	Hence the hyperplane condition fails
	and $g(q,\{x_1,\ldots,x_d\})=0$. 

	If $e_d \in \partial {\rm Cone}(u_1,\ldots,u_d)$,
	let $H$ be
	a supporting hyperplane to 
	 ${\rm Cone}(u_1,\ldots,u_d)$ with $e_d \in H$. Then $o \in H$.
	 Let $H'=H + q$. Then $q, q+ e_d \in H'$ and
	 all points of $\{x_1,\ldots,x_d \} \setminus H'$ lie on the
	 same side of $H'$. Thus the hyperplane condition fails
	 so again  $g(q,\{x_1,\ldots,x_d\}) =0$.
	 Thus we have verified the second equality of \eqref{e:hgg}.
\end{proof}

\begin{lemma}
	\label{l:covUB}
	Let $\eps $ be as in the statement of
	Theorem \ref{t:main}.
	As $t \to \infty $ we have
	that
		$
		(\PP[R_{t,k}\le r_t] - \PP[F_t(A^o) =0] ) \vee 0
		= O (t^{-\eps \xxi/2} ). 
	$
\end{lemma}

\begin{proof}
	Suppose $F_t(A^o) > 0$. 
	Then we can find distinct $\bx_1,\ldots,\bx_d\in \xi'_t$ with 
	$h^{(r_t)}(\bx_1,\ldots,\bx_d)=1$ such that
	(setting $\bp= \bp^{(r_t)}(\bx_1,\ldots,\bx_d)$ and
	 $\bq= \bq^{(r_t)}(\bx_1,\ldots,\bx_d)$) we have
	$q \in V_k(\xi'_t \setminus\{\bx_1,\ldots,\bx_d\},r_t)$ and
	$\bq\in A^o$. 
	If moreover $q \in V_1(\xi''_t,r_t)$,
	then
	we have for all all small enough $b >0$ that $q + b (q-p)$ lies
	in $V_k(\xi_t,r_t) \cap A$, so that  $R_{t,k} >r_t$.

	For any point $q \in \RR^d$, by the Markov and H\"older
	inequalities, using the assumption
	$\EE[Y^{d+\eps} ]< \infty$,
	we have that
	\begin{align}
		\PP [q \in Z_1(\xi''_t,r_t)] & \leq t \int_{(t^\xxi,\infty)} 
		\lambda_d( B(q,r_ta))
		\QQ(da) = \theta_d t r_t^d \EE[Y^{d}
		{\bf 1} \{Y > t^\xxi \}]
		\nonumber \\
		& \leq 
		\theta_d t r_t^d (\E[Y^{d+\eps}])^{d/(d+\eps)}
		(\PP [ Y \geq t^\xxi])^{\eps/(d+\eps)} 
		\nonumber \\
		& \leq 
		\theta_d t r_t^d (\E[Y^{d+\eps}])^{d/(d+\eps)}
		(\E[Y^{d+\eps}]/t^{\xxi(d+\eps)})^{\eps/(d+\eps)}.
		\label{e:fromHolder}
	\end{align}
	Therefore using the independence of $\xi'_t$ and $\xi''_t$, we have that
	\begin{align*}
		\PP [F_t(A^o) >0, R_{t,k} \leq r_t]
		\leq 
		\theta_d t r_t^d \E[Y^{d+\eps}]
		t^{- \eps \xxi} =
		O( t^{- \eps \xxi/2} ),
	\end{align*}
	and the result follows.
\end{proof}

The inequality the other way is more delicate,
and will require some preparation.
We first prove a result
about the geometry of intersections of spheres with centres 
in general position (Corollary \ref{l:genpos} below).
A result along these lines is also seen in \cite[Lemma 4.1]{Janson}.

Given $\ell \in [d]$, we define an $\ell$-dimensional sphere to be a set of
the form $\partial B(x,r) \cap H$, for some $x \in \R^d$, some 
$r>0$, and some
$(\ell+1)$-dimensional affine subspace $H$ of $\R^d$ containing $x$.

\begin{lemma}
	\label{l:gen0}
	Suppose $a_1,\dots, a_{d+1} \in (0,\infty)$.
	 Let $\ell \in [d+1]$. 

	(i) If $\ell \leq d$ then for
	$\lambda_d^\ell$-almost all $(x_1,\ldots,x_\ell) \in (\R^{d})^\ell$
	the set $ S_\ell:= S_\ell(x_1,\ldots,x_\ell) := \cap_{i=1}^\ell \partial B(x_i,a_i)$ is either
	a $(d-\ell)$-dimensional sphere or the empty set.

	(ii) For $\lambda_d^{d+1}$-almost
	all $(x_1,\ldots,x_{d+1}) \in (\R^{d})^{d+1}$
	we have $S_{d+1} = \emptyset$.
\end{lemma}
\begin{proof}
	The result is clear for $\ell=1$.
	 Suppose $\ell=2$. Set $D= \|x_1-x_2\|$.  

Suppose also $d=2$.
	Then $S_2 = \emptyset $ if $D > a_1 + a_2 $ or 
	$D < |a_1-a_2|$. If $|a_1-a_2 | < D < a_1 + a_2$ then
	$S_2$ has two points and they are mirror to each other
	in the line containing $x_1,x_2$; this is geometrically obvious 
	and  can be verified in the special case where
	$x_1 =(0,0)$ and $x_2 = (D,0)$ by solving the two
	simultaneous equations for the two coordinates
	of any point of intersection.

	Now suppose $d \geq 3$ and $\ell =2$. Suppose
	 $|a_1-a_2 | < D < a_1 + a_2$, and assume without
	 loss of generality that $x_1 =o$ and $x_2 = De_1$,
	 where we set $e_1:=(1,0,\ldots,0)$.
	 Given $u \in \R$ and $z \in \R^{d-1}$, we have
	 $$
	 (u,z) \in S_2 \Longleftrightarrow (u,\|z\|,0\ldots,0) \in S_2.
	 $$
	 Therefore picking $u_0 \in \R,v >0 $ such that
	 $(u_0,v,0,\ldots,0) \in
	 S_2$ (by the case $d=2$ already considered there
	 is just one such choice of $u_0,v$),
	 we have that $S_2 $ is the intersection
	 of $\partial B(u_0e_1,v) $ with the hyperplane perpendicular
	 to $e_1$ through $u_0e_1$, and therefore is a $(d-2)$-dimensional
	 sphere.

	 Now suppose $d \geq \ell =3$. Then $S_3 = S_2 \cap \partial B(x_3,a_3)$
	 and by the case already considered, for almost all $(x_1,x_2)$,
	 $S_2$ is either empty or a $(d-2)$-dimensional sphere contained
	 in a $(d-1)$-dimensional affine subspace $H$ of $\RR^d$. 
	 In the latter case, since
	 $\partial B(x_3,a_3) \cap H$ is (for $\lambda_d$-almost all $x_3$)
	 either empty or a $(d-2)$-dimensional sphere
	 in $H$, by the case $\ell=2$ its intersection with  $S_2$ is,
	 for almost all choices of $x_3$, either a $(d-3)$-dimensional
	 sphere or the empty set.

	 Continuing in this way (inductively in $\ell$) we obtain the result
	 (i). Part (ii) then follows
	 from the fact that $S_d$ has at most two points for
	 $\lambda_d^d$-almost all choices of $(x_1,\ldots,x_d)$.
\end{proof}

\begin{corollary}\label{l:genpos}
 Let $r>0$, $\ell\in [d]$.
	 Almost surely, for all $\ell$-tuples
	 of distinct $\bx_1 =(x_1,a_1), \ldots , \bx_\ell =(x_\ell,a_\ell)
	 \in \xi_t$ and 
	satisfying $S_\ell \neq \emptyset$, where we set  
$$
	S_\ell:=S_\ell(\bx_1,...,\bx_\ell) := \cap_{i=1}^\ell  \partial B(x_i, ra_i),
	$$
	the set $S_\ell$ is a $(d-\ell)$-dimensional sphere. Also,
	almost surely, $S_{d+1}(\bx_1,...,\bx_{d+1})=\emptyset$ for distinct $\bx_1,...,\bx_{d+1}\in\xi_t$. 
\end{corollary}
\begin{proof} 
	This  follows from Lemma \ref{l:gen0}
	and application of the multivariate Mecke formula (see e.g.
	\cite{LP}) and Fubini's theorem.
\end{proof}

Next we use a percolation argument to show that with high probability,
all components of $V_k(\xi'_t,r_t)$ intersecting $A$ are confined
to the set $A' := A'(t):= A \oplus B(o,\sqrt{r_t})$.
\begin{lemma}
	\label{l:perc}
	Given $t >0$, let $\scr V_{k,t}$
	be the event that every connected component of the vacant set 
	$V_k(\xi'_t, r_t)$ that  intersects $A$
	is contained in $A'$.
	Then there exists $c >0 $ such that for all large enough
	$t$, $\PP[\scr V_{k,t}]\ge 1 - \exp(-c t^{1/(3d)})$.
\end{lemma}

\begin{proof}
	Choose $\delta >0$ such that 
	$\PP[Y >  \sqrt{d} \delta  ] > 2 \delta$.
	Given $t$, partition $\RR^d$ into cubes
	of side length $\delta r_t$ indexed by $\ZZ^d$ such
	that for each $z \in \ZZ^d$, the
	cube indexed by $z$ is centred on $\delta r_t z$.
	We say a collection of cubes is a {\em path}
	if the indexing lattice points make a path
	in the graph $(\ZZ^d,\sim)$ where for
	$x,y \in \ZZ^d$ we put $x \sim y$ if and only if 
	$\|x-y\|_\infty=1$ (i.e., allow diagonal connections).  
	If $\scr V_{k,t}$ does not occur, then  there exists a
	path of cubes
	such that
\begin{itemize}
\item[1)]  at least one constituting cube $Q$ of the path intersects 
	$A$;
\item[2)]  the cardinality of the path
	is at least $n(t) :=
	\lceil \sqrt{r_t} / (\sqrt{d} \delta r_t) 
		\rceil = \lceil r_t^{-1/2}/(\sqrt{d}\delta ) \rceil$;
\item[3)] each constituting cube $Q$ of the animal satisfies $Q \cap
	V_k(\xi'_t,r_t) \neq \emptyset$.
\end{itemize}
	Item 3) implies that $\#(\xi'_t \cap (Q\times [\sqrt{d}\delta ,\infty)]))
	\leq (k-1)$
	because otherwise $Q\cap V_k(\xi'_t,r_t) =\emptyset$. 
	Observe that provided $t$ is sufficiently large,
	$t (\lambda_d \otimes \QQ)(Q \times [
		\sqrt{d}\delta,t^\xxi]) \geq t \delta  (\delta r_t)^d $,
	and thus by e.g. \cite[Lemma 1.2]{Pen}
	$$
	\PP[ \#(\xi'_t \cap (Q \times [\sqrt{d}\delta,\infty))) \leq k-1]
	\leq \exp( - (t/2)\delta^{d+1} r_t^d).
	$$
	For all $n \in \NN$
	the number of paths of cardinality $n$ starting from the origin
	is at most $3^{dn}$.
	Hence there is a constant $c'$ such that
	the number of paths satisfying 1) and 2)
	is at most $c' r_t^{-d} 3^{dn(t)}$,
	and hence for $t$ large is at most
	$t 3^{dn(t)}$.
	Provided $t$ is also large enough so that
	$\delta^{d+1} t r_t^d > 4 \log (3^d)$ and
	$r_t^{-1/2} /(d^{1/2} \delta) \geq t^{1/(3d)}$,
	applying the union bound leads to
\begin{align*}
	\PP[\scr V_{k,t}^c]
	& \le
	t 
	3^{dn(t)}
	\exp(- n(t) (t/2) \delta^{d+1}  r_t^{d} )
	\\
	& \leq t 
	\exp (- (\log 3^d) n (t))
	\\
	& \le t \exp( - (\log 3^d) t^{1/(3d)} ), 
\end{align*}
	and the result follows.
\end{proof}


We are now ready to provide an inequality in the opposite direction
to Lemma \ref{l:covUB}.

\begin{lemma}\label{l:sc_cov}
	For all $t$,
	it holds that $\PP[\{R_{t,k} > r_t\} \cap \scr V_{k,t}]
	\leq \PP[F_t(A')  > 0].
	$
\end{lemma}

\begin{proof}
	Suppose that event $\scr V_{k,t}$ occurs, and $R_{t,k} >r_t$.
	Then we can and do choose $z \in V_k(\xi_t,r_t) \cap A
	\subset V_k(\xi'_t,r_t) \cap A$.
	Let $V_{z}$ denote the 
	closure of the connected component of the vacant set
$V_k(\xi'_t,r_t)$ containing $z$. Then, since
	we assume $\scr V_{k,t}$ occurs,  
	$V_z \subset A'$.
	
	Let $\bq$ denote a lowest point of $V_z$, i.e.
	$\langle \bq ,e_d \rangle \le \langle x,e_d \rangle$ 
	for all $x \in V_z$. 
	Then $q \in A'$.

For any simple counting measure $\mu$ on $\RR^d\times \R_+$ and $r>0$, we
	define  SPBM$(\mu,r)$ to be the collection of all
	balls of the form $B(x,r a)$ with $(x,a) \in \mu$.

We claim that almost surely, $\bq$ lies on the boundary of exactly $d$ balls in 
$\mathrm{SPBM}(\xi'_t,r_t)$, and in the interior of exactly $k-1$ of these
	balls. 
	It is clear that $\bq \in B^o$ for at most $k-1$ balls $B\in
	\mathrm{SPBM}(\xi'_t,r_t)$ and $\bq\not\in V_k(\xi'_t,r_t)$;
	otherwise  either $\bq$ lies in the interior of the $k$-occupied set
	$Z_k(\xi'_t,r_t)$ or $\bq$ is not a lowest point of $V_z$.
	Hence $\bq$ has to lie on at least one sphere. 

Suppose  that $\bq$ lies on  exactly $m$ spheres with $1\le m \le d-1$. In particular, it lies on the hyperplanes tangent to the spheres at $\bq$. Let $I$ denote the intersection of these hyperplanes. Then the dimension of $I$ is at least 1. For some small $\ep>0$, any point in $(B(\bq,\ep)\cap I) \setminus \{\bq\}$ 
	lies in the interior of $V_z$. Necessarily there exists a point 
	$\bq' \in B(q,\eps ) \cap I$ that is as low as $\bq$ such that for some
	$\ep'>0$, we have $B(\bq',\ep')\subset V_z$. 
	It follows that the point $\bq' - \ep' e_d/2\in V_z$ is strictly lower than $\bq$, a contradiction. Also, by Corollary \ref{l:genpos} 
	the intersection of  $d+1$ distinct spheres is empty.
	Therefore $\bq $ lies on the boundary of exactly $d$
	balls in SPBM$(\xi'_t,r_t)$, denoted $B_1,\ldots,B_d$ say.
	Moreover, if $q$ lies in the interior
	of fewer than $k-1$ balls, denoted
	say $B'_1,\ldots,B'_{\ell}$ with $\ell < k-1$,
	then by repeating the above argument
	 for $m=d-1$ 
	we see there is  a point $q'$ near $q$ and lower than $q$ that
	lies in none of $B_1,\ldots, B_{d-1}$ and therefore in the interior
	of at most $k-1$ balls, namely $B'_1, \ldots, B'_\ell$ and
	$B_d$. This gives a contradiction, which
	 completes the proof of the claim.

By Corollary \ref{l:genpos} with $\ell =d$ and $r=r_t$,
	there exist distinct 
	$\bx_1,...,\bx_d\in\xi_t$ with $S_d(\bx_1,...,\bx_d)=\{\bq, \bp\}$ where $\bp$ is the mirror of $\bq$ with respect to the affine
	space generated by $\bx_1,...,\bx_d$ by symmetry of balls.
	For $i \in [d] $ we write $\bx_i = (x_i,a_i)$ with
	$x_i \in \R^d$ and $a_i \in \R_+$.
	Observe that
$q \in V_k(\xi'_t \setminus \{\bx_1,\ldots,\bx_d\})$, since
$q$ lies on exactly $d$ spheres.

We claim that $\bp$ is strictly lower than
	$\bq$.  Indeed, by convexity of balls, the segment $[\bp,\bq]$ is contained in the ball associated with $\bx_j$ for all $j \in [d]$,
	and any other point on the line containing $[\bp,\bq]$ is not covered by any of the $d$ balls.  
	If $p$ were not strictly lower than $q$, then for
	all sufficiently small 
	$\delta >0$, $q+ \delta (q-p)  \in V_k(\xi'_t,r_t)$
	and taking a further point just
	below $q+ \delta (q-p)$ would give  us a point in $V_z$ 
	that is lower than $\bq$, and hence a contradiction.

	Next we claim that $h^{(r_t)}(\{\bx_1,\ldots,\bx_d\})=1$.
	Indeed, since $q \in V_k(\xi_t \setminus \cup_{i=1}^d \bx_i)$
	and $q$ is a lowest point of $V_z$, $q$  must be a local minimum
	of the set $\RR^d \setminus \cup_{i=1}^d B(x_i, r_t a_i)^o$.
	Here we use the fact that for all small $\delta >0$ the
	set $B(q,\delta) \setminus \cup_{i=1}^d B(x_i,r_t a_i)^o$ is connected.

	It follows from all the above that
 $F_t(A') >0$, since the set $\{\bx_1,\ldots,\bx_d\} $ contributes
 to the sum on the right hand side of \eqref{e:Ftdef}.
\end{proof}

\section{Convergence of expectation}
\label{s:convexp}

In this section we prove the convergence of the expectation of 
$F_t(D)$ (defined at \eqref{e:Ftdef}) as $t\to\infty$,
for bounded Borel $D \subset \R^d$ (allowed to vary with $t$).
Recall
from Section \ref{secintro}
that we set $\alpha := \theta_d \EE[Y^d]$.
Given $(a_1,a_2, \ldots,a_d) \in \R_+^d$, define
\begin{align}
	G(a_1,\ldots,a_d)
	:= \frac{1}{d!}
	\int_{\RR^{d(d-1)}} h^{(1)}(\{(o,a_1),(x_2,a_2),\ldots,(x_d,a_d)\})
	\lambda_d^{d-1}(d(x_2,\ldots,x_d)),
	\label{e:cdef}
\end{align} 
and define the constant
	  $c_0 := \int_{\R_+^d} 
	  G(a_1,\ldots,a_d)
	  \prod_{i=1}^d\QQ(da_i)$ (this constant depends
	 on $d$ and $\QQ$).

\begin{proposition} 
	\label{p:meanlim}
	Suppose for all $t >1$ that $D(t) \subset \R^d$
	is Borel-measurable with
	$\lambda_d(D(t)) =O(1)$ as $t \to \infty$.
	then as $t \to \infty$,
	\begin{align}
		\EE[F_t(D(t))] -  c_0 
		\alpha^{1-d} ((k-1)!)^{-1}
		\lambda_d(D(t)) e^{-\beta} =O
		\Big(\frac{1}{\log t}\Big).
		\label{e:limEF}
		\end{align}
\end{proposition}
 \begin{proof}
	 Notice that for all $x \in \RR^d$ we have
	 that $\xi'_t(\{(y,b): x \in B(y,r_t b)\})$
	 is Poisson distributed with mean 
	 $ t \int_{[0,t^\xxi]} \theta_d (br_t)^d \QQ(db)$,
	 and by the calculation at \eqref{e:fromHolder},
	 $ t \int_{[0,t^\xxi]} \theta_d (br_t)^d \QQ(db)
	 = \alpha tr_t^d + O(t^{-\eps \xxi/2})$,
	 implying
	 $\PP[x \in V_k(\xi'_t,r_t)] = ((\alpha t r_t^d)^{k-1}/(k-1)!) 
	 e^{-\al t r_t^d}(1
	 + O((\log t)^{-1}))$,
	 and this probability does not depend on $x$.
	 Therefore
	 by \eqref{e:Ftdef} and
	 the multivariate Mecke formula, writing just $D$ for $D(t)$,
	 we have 
 \begin{align} 
	 \EE[F_t(D)] = \frac{t^d
	 e^{-\alpha t r_t^d}
	 (\alpha t r_t^d)^{k-1}(1 + O(\frac{1}{\log t})
	 )
	 }{d!(k-1)!} 
	 \int_{\R_+^d} \prod_{i=1}^d\QQ(da_i)
	 \int_{(\R^d)^d}
	 \nonumber\\
	 h_D^{(r_t)}(\{(x_1,a_1),\ldots,(x_d,a_d)\})
	 \lambda_d^{d}(d(x_1,\ldots,x_d)) 
	 .
	 \label{e:0415b}
 \end{align}

It is easy to see that for $a_1,\dots,a_d$ fixed,
	 \begin{align*}
		 &\int_{(\R^d)^d} h_D^{(r_t)}(\{(x_1,a_1),\ldots,(x_d,a_d)\}) 
		 \lambda_d^{d}(d(x_1,\ldots,x_d))
		 \\ & = 
		 \la_d(D)\int_{\RR^{d(d-1)}} h^{(r_t)}(\{(o,a_1),(x_2,a_2),
		 \ldots,(x_d,a_d)\}) 
		 \lambda_d^{d-1}(d(x_2,\ldots,x_d)),
	 \end{align*}
	 and moreover
	 \begin{align*}
		 h^{(r_t)}(\{(o,a_1), (x_2,a_2),\ldots , (x_d,a_d)\})
		 = h^{(1)}(\{(o,a_1), (r_t^{-1} x_2,a_2),\ldots, 
		 (r_t^{-1} x_d,a_d)\}).
	 \end{align*}
	 Therefore a change of variables and \eqref{e:cdef} gives us
\begin{align*}
	&\int_{(\R^d)^d} h_D^{(r_t)}(\{(x_1,a_1),\ldots,(x_d,a_d)\})
	\lambda_d^d(d(x_1,\ldots,x_d))
	= d! r_t^{d(d-1)} \lambda_d(D) G(a_1,\ldots,a_d).
\end{align*}
	 Then using the definition just after \eqref{e:cdef} of $c_0$, and 
	  \eqref{e:0415b}
	  followed by \eqref{e:r_t}
	 we obtain that for $t$ large,
\begin{align*}
	\EE[F_t(D)] & = 
	t^d
	e^{-\al t r_t^d} 
	r_t^{d(d-1)} \la_d(D) (\alpha t r_t^d)^{k-1} ((k-1)!)^{-1}
	c_0 \Big(1+ O \big( \frac{1}{\log t}\big) \Big) \\
	 & =
	 \Big( \frac{
	 c_0 \lambda_d(D) e^{-\beta}  
		 t^{d-1}
	 (\alpha t r_t^d)^{k-1}}{ (k-1)! (\log t)^{d+k-2}}
	\Big)
	\exp \Big(- \frac{(d+k-2)^2 \log \log t}{
		\log t} \Big) 
		\Big( 1 + O\big(\frac{1}{\log t}\big) \Big)
		\\
	& \quad \quad
	\times \Big( \frac{\log t + (d+k-2) \log \log t + \beta + (d+k-2)^2
		(\log \log t)/\log t}{ \alpha  t }\Big)^{d-1}
		\\
		&  =
		c_0 \lambda_d(D) e^{-\beta} ((k-1)!)^{-1} \alpha^{1-d}
		\Big( 1+ O \big( \frac{1}{\log t} \big) \Big),  
\end{align*}
as required.
 \end{proof}
	 
The proof of Theorem \ref{t:main} in the next section will
in fact show that $\PP[ R_{t,k} \leq r_t] \to \exp (-c_0 \alpha^{1-d}
\lambda_d(A) e^{-\beta})$ as $t \to \infty$. Thus 
	for consistency with the known result of Theorem \ref{t:H-J}
	the constant $c_0$ must satisfy 
	$ (c_0 \alpha^{1-d})/(k-1)! = c_{d,k,Y}, $
	where $c_{d,k,Y}$ was given at \eqref{e:cdkY}.
	i.e. 
	\begin{align}
	c_0 =  
		\frac{ (\E[Y^{d-1}])^{d}
		\theta_d^{d-1}}{d!
		} 
	\Big( \frac{\sqrt{\pi} \Gamma(1+d/2)}{\Gamma((1+d)/2)} \Big)^{d-1}
	=
		\frac{ (\E[Y^{d-1}])^{d} \pi^{(d^2-1)/2}
		}{d! (\Gamma((1+d)/2))^{d-1}}. 
		\label{e:c0}
	\end{align}
Nevertheless, to make this paper more self-contained,
we sketch an argument to verify \eqref{e:c0} directly.
	\begin{lemma}[Computation of $c_0$]
		\label{l:compc0} Given
		$s_1,\ldots,s_d \in \R_+$, 
		it holds that
		\begin{align}
			G(s_1,\ldots,s_d) =
		\frac{ \pi^{(d^2-1)/2} }{d! (\Gamma((1+d)/2))^{d-1}}
			\prod_{i=1}^d
			s_i^{d-1},
			\label{e:G}
		\end{align}
			and
		\eqref{e:c0} holds.
	\end{lemma}

	\begin{proof}
		Clearly it suffices to prove \eqref{e:G}, which
		implies \eqref{e:c0}.

First suppose $d=2$.	Fix $s_1,s_2 >0$. 
	For $0< u < \pi/2, 0< v < \pi/2$
        consider
	an infinitesimal portion of $\partial B(o,s_1)$   at angle between $u$
        and $u+du$ clockwise from the top,
	and an infinitesimal portion of $\partial B(x_2,s_2)$ at angle between $v$ and $v+dv$
	anticlockwise from the top.

	Since the two portions are at an angle $(u+v)$ to each other,
	the integral over $\lambda_2(dx_2)$ of the indicator that
	these portions cross (making a local minimum of
	the set $\R^2 \setminus (B(o,s_1)^o \cup B(x_2,s_2)^o)$)
	is $s_1 s_2 \sin (u+v) dx dy$

        Also for $0 < v < u < \pi/2$ consider
	a portion of $\partial B(o,s_1)$   at angle between $u$
        and $u+du$ clockwise from the  {\em bottom},
	and a portion of $\partial B(x_2,s_2)$ at angle between $v$ and $v+dv$ 
	clockwise from the top.
	The integral over $\lambda_2(dx_2)$ of the indicator that
	the portions cross (again making a local minimum of
	the set $\R^2 \setminus (B(o,s_1)^o \cup B(x_2,s_2)^o)$)
	is
        $s_1 s_2 \sin (u-v)$. Hence
	\begin{align*}
		\int_{\R^2} h^{(1)} (\{(o,s_1),(u,s_2) \} ) \lambda_2(du) 
		& =
        2 s_1 s_2 \int_0^{\pi/2} \int_0^{\pi/2} \sin (u+v) dv du
		\\
		& \quad +
        4 s_1 s_2 \int_0^{\pi/2} \int_0^u \sin (u-v) dv du
		\\
		& = 4  s_1 s_2 \Big(1+ \int_0^{\pi/2} (1- \cos u) du \Big)
         = 2 \pi s_1 s_2.
	\end{align*}
		Thus using the definition \eqref{e:cdef} we obtain that
		$G(s_1,s_2) = \pi s_1s_2$, which gives us the
		case $d=2$ of \eqref{e:G}.

	The above argument uses infinitesimals but can be made rigorous
	by approximating the two circles $\partial D_1$ and $\partial D_2(u)$
	with regular polygons with $m$ sides  and taking the limit
	$m \to \infty$. We omit the details.

	Next we 
	verify  \eqref{e:G} in general dimension $d \geq 3$.
	Fix $s_1,\ldots,s_d \in (0,\infty)$. For $i \in [d]$
	and $x_i \in \R^d$
	consider a portion  of $\partial
	B (x_i,s_i)$ at locations $x_i + s_i e$, $e \in du_i$,
	where $du_i$ is an infinitesimal $(d-1)$-dimensional
	volume element of  the
	$(d-1)$-dimensional unit sphere $\mathbb{S}_{d-1}$.

	The integral over $\lambda_d(dx_2) \cdots \lambda_d(x_d)$
	of the indicator that these  portions have a mutual point
	of intersection with
	each other comes to 
	$|{\rm Det}(u_1,\ldots,u_d) | \prod_{=1}^d (s_i^{d-1} du_i)$,
	and the absolute value of the 
	determinant is unchanged if we replace any $u_i$
	with $-u_i$.

	By Lemma \ref{l:hgg},
	for the point of intersection of these portions to make
	a local minimum of the set $\R^d \setminus \cup_{i=1}^d B(x_i,s_i)^o$
	it necessary and sufficient  that
	$ e_d \in {\rm Cone}(u_1,\dots,u_d)^o$. 
	If we fix $u_1,\ldots,u_d$ in general position
	on $\mathbb{S}_{d-1}$ then we can express $e_d$ uniquely
	as a linear combination
	of $u_1,\ldots,u_d$, and if we are free to reverse  the direction
	of
	$u_1,\ldots,u_d$ then for exactly one out of the  $2^d$ choices
	of the directions for $u_1,\ldots,u_d$ this linear combination
	will have positive coefficients. Thus
	\begin{align*}
		\int_{\RR^{d(d-1)}} h^{(1)}(
		\{(o,s_1),(x_2,s_2),\ldots,(x_d,s_d)\})
	\lambda_d^{d-1}(d(x_2,\ldots,x_d))
		\\
		= 2^{-d} \Big( \prod_{i=1}^d s_i^{d-1} \Big)
		\int_{\mathbb{S}_{d-1} } 
		\cdots
		\int_{\mathbb{S}_{d-1} } 
		|{\rm Det}(u_1,\ldots,u_d) | du_d \cdots du_1.
	\end{align*}
	By \cite[Lemma 9.3]{Janson} the last expression comes to
	\begin{align*}
		2^{-d} \Big( \prod_{i=1}^d s_i^{d-1} \Big)
		 (d \theta_{d})^{d}
		 \pi^{-1/2} \Gamma(d/2)^d \Gamma((d+1)/2)^{1-d}. 
	\end{align*}
	Therefore using the formula for $\theta_d$ and \eqref{e:cdef}
	we obtain that
	\begin{align*}
		G(s_1,\ldots,s_d) = \frac{d^d
		\big(\prod_{i=1}^d s_i^{d-1} \big)
		\pi^{d^2/2}}{d! 2^d 
		\Gamma(1+d/2)^d}
		\times
		\frac{\pi^{-1/2}\Gamma(d/2)^d}{\Gamma((d+1)/2)^{d-1}}
		= \frac{
			\big(\prod_{i=1}^d s_i^{d-1} \big)
		\pi^{(d^2-1)/2}}{d!
		\Gamma((d+1)/2)^{d-1}},
	\end{align*}
	confirming \eqref{e:G}.
	\end{proof}

\section{Proof of Theorem \ref{t:main}}

To complete the proof, given $t >1$ we partition $\R^d$ into rectilinear cubes
$Q_{t,i}, i \in \ZZ^d$, of side $\sqrt{r_t}$. Let $\cI(t)$ be the
collection of $i \in \ZZ^d$ such that $Q_{t,i} \subset A$. 
For $i \in \ZZ^d$, let $Q_{t,i}^-$ be a slightly smaller
rectilinear cube of side $\sqrt{r_t} - 2t^\xxi r_t$ with the same centre
as $Q_{t,i}^-$.
Later we shall prove the following result:
\begin{proposition}
	\label{p:sumbd}
	It is the case that
\begin{align}
	\label{e:Esumbd}
	\sum_{i \in \cI(t)} \EE [ F_t(Q_{t,i}^-) (F_t(Q_{t,i}^-)-1) ] 
	= O \big( \frac{1}{\log t} \big) ~~~~~~~
	{\rm as}~ t \to \infty.
\end{align}
\end{proposition}

\begin{proof}[Proof of Theorem \ref{t:main}]
	\blue{To aid intuition,
		we give a sketch
		before going into details.
		The idea is to cover most of $A$ by boxes
	$Q_{t,i}^-, i \in {\cI(t)}$, and then approximate
	the event $\{R_{t,k} \leq r_t\}$ by the event
	$\cap_{i \in \cI(t)} \{F_t(Q^-_{t,i})=0\}$,
	which is an intersection of independent events each with
	probability close to 1.
	We shall then use Proposition \ref{p:sumbd} to show that
	$\Pr[F_t(Q^-_{t,i}) \geq 1]$ is well approximated 
	by $\E[F_t(Q^-_{t,i})]$.  Thus we
	can approximate $\sum_{i \in \cI(t)} \Pr[F_t(Q^-_{t,i}) \geq 1]$
	by $\E[ F_t(\cup_{i \in \cI(t)} Q_{t,i}^- )] $, which in turn
	can be handled using Proposition \ref{p:meanlim}.

Here are the details.}
	Recall that $A':= A \oplus B(o,\sqrt{r_t})$ and
	$A'':= A''(t):= \{y \in A: B(y,\sqrt{d r_t}) \subset A\}$.
	Then $A'' \subset \cup_{i \in \cI(t)}Q^-_{t,i}$, and
\begin{align}
	\lambda_d ( A' \setminus \cup_{i \in \cI(t)} Q_{t,i}^- ) 
	& \leq \lambda_d(A' \setminus A'') + 
	\sum_{i \in \cI(t)} r_t^{d/2}(1-(1-2t^\xxi r_t/r_t^{1/2})^d)
	\nonumber \\
	& = O(\lambda_d(A' \setminus A'') + t^\xxi r_t^{1/2} )
	\label{e:AminQ}
\end{align}
	Since we took $\xxi < 1/(2d)$ at \eqref{e:xxi}, we have that
	$t^\xxi r_t^{1/2} = O((\log t)^{-1})$.
Therefore by Markov's inequality and Proposition \ref{p:meanlim}, 
	we have that
\begin{align}
	\PP[F_t(A' \setminus \cup_{i \in \cI(t)} Q_{t,i}^-) \geq 1 ]
	\leq
	\E[F_t(A' \setminus \cup_{i \in \cI(t)} Q_{t,i}^-) ]
	 = O \Big(
	 \lambda_d(A' \setminus A'') 
	 + \frac{1}{\log t} 
	 \Big).
	 \label{EN_tbdy}
\end{align}
Next we use the fact that for any number $x \in \NN \cup \{0\}$
we have $x - {\bf 1 }\{x \geq 1 \} \leq x(x-1)/2.$ Applying
	this inequality to the random variables $F_t(Q_{t,i}^-)$, $i \in \cI(t)$,
we obtain that
\begin{align}
	\EE \Big[ \sum_{i \in \cI(t)} ( F_t(Q_{t,i}^-) - {\bf 1}\{F_t(Q_{t,i}^-)
	\geq 1\}  ) \Big] \leq \sum_{i\in \cI(t)} \EE [
		F_t(Q_{t,i}^-) (F_t(Q_{t,i}^-)-1)/2 ]. 
	\label{e:Esum1}
\end{align}

Set $\mu_t := \sum_{i \in \cI(t)}
	\PP[ F_t(Q_{t,i}^-) \geq 1]$.
Using \eqref{e:Esum1}
	 with
	Proposition \ref{p:sumbd}, and then 
	Proposition \ref{p:meanlim} and \eqref{e:AminQ}, gives us that
\begin{align}
	\mu_t
	& = \EE \Big[ \sum_{i \in \cI(t)} F_t(Q_{t,i}^-) \Big]
	+ O \Big( \frac{1}{\log t} \Big) 
	\nonumber \\
	& = c_0 \alpha^{1-d} ((k-1)!)^{-1} \lambda_d(A) e^{-\beta}
	+ O \Big( \lambda_d(A \setminus A'') 
	+ \frac{1}{\log t} \Big).
	\label{e:mu_t}
\end{align}

	Let $p_t := \PP[F_t(Q_{t,i}^-) \geq 1]$, which is the same for
all $i \in \cI(t)$.
	Since the variables $F_t(Q_{t,i}^-), i \in \cI(t)$, are 
independent,
\begin{align}
	\Pr[ F_t(\cup_{i \in \cI(t)} Q_{t,i}^-) =0]
= (1-p_t)^{\#(\cI(t))} \leq 
\exp(-p_t \#(\cI(t))) = e^{-\mu_t}.
	\label{e:expmu}
\end{align}
	By Lemma \ref{l:covUB}, \eqref{e:expmu}
	and \eqref{e:mu_t},
	\begin{align}
		\PP[R_{t,k} \leq r_t] & 
		\leq \PP[F_t(A^o) =0]
		+ O ( t^{-d \xxi/2} ) 
		\nonumber \\
		& \leq \PP[F_t(\cup_{i \in \cI(t)} Q_{t,i}^- ) =0]
		+ O ( t^{-d \xxi/2} ) 
		\nonumber \\
		& \leq  \exp (-
	 c_0 \alpha^{1-d} ((k-1)!)^{-1} \lambda_d(A) e^{-\beta})
		+ O \Big( \lambda_d(A \setminus A'') + \frac{1}{\log t}
		\Big).   
		\label{e:PRUB}
	\end{align}

For an inequality the other way, observe that
$p_t = O(1/\#(\cI(t)) ) = O(r_t^{d/2})$,
and for $t$ large
$1-p_t \geq \exp( - p_t -p_t^2)$
so that
$$
\Pr[ F_t(\cup_{i \in \cI(t)} Q_{t,i}^-) =0]
\geq \exp ((- p_t-p_t^2) \#(\cI(t))) =
(1+ O(r_t^{d/2}))
e^{-\mu_t }.
$$
Then using \eqref{EN_tbdy} and \eqref{e:mu_t} we obtain that
\begin{align}
\Pr[ F_t(A') =0] \geq \exp (-
	 c_0 \alpha^{1-d} ((k-1)!)^{-1} \lambda_d(A) e^{-\beta})
	+ O \Big( \lambda_d(A' \setminus A'')   
	+ \frac{1}{\log t} \Big).
	\label{e:PFLB}
\end{align}
	By rearranging the inequality given in 
	Lemma \ref{l:sc_cov}, using Lemma \ref{l:perc} 
	and \eqref{e:PFLB},
\blue{we obtain
	that}
	\begin{align*}
		\PP[R_{t,k} \leq r_t] 
		& \geq \PP[F_t(A') =0] - \PP[ \scr V_{k,t}^c]
		\\ 
		& \geq  \exp (-
	 c_0 \alpha^{1-d} ((k-1)!)^{-1} \lambda_d(A) e^{-\beta})
		+ O \Big( \lambda_d(A' \setminus A'') + \frac{1}{\log t}
		\Big).   
	\end{align*}
	Together with \eqref{e:PRUB}, this 
	yields \eqref{e:JQ} as required.
\end{proof}

	We still need to prove Proposition \ref{p:sumbd}.
Set $\alpha := \theta_d  \EE [Y^d]$ as before.
Let $\XX := \R^d \times [0,\infty)$. 
Given $t>1$, for $ x \in \R^d$ \blue{we set} 
\begin{align}
	B(x):= \{(y,b) \in \R^d \times \R_+: \|y-x\| \leq r_t b\}.
	\label{e:defBx}
\end{align}
Also, given $\bar\bx = \{\bx_1,\ldots,\bx_d\} \subset \XX$, we write
$q_{\bar\bx} $ for $q^{(r_t)}(\bar\bx)$.
Moreover, from now on we shall use $|\cdot|$ rather than
$\#(\cdot)$ to denote the number of elements of
a finite set in $\XX$.

\begin{lemma}\label{l:simple}
	Let $D $ be a bounded Borel set in $\R^d$, and $t >1$.
	With $F_{t}(D)$ defined at
	\eqref{e:Ftdef},
	we have
\begin{align*}
	F_{t}(D)(F_{t}(D) -1)  =
	\sum_{\bar\bx \subset \xi'_t}
	h_D^{(r_t)}(\bar\bx)
	(\II^{(1)}_{\bar\bx,D}+
	\II^{(2)}_{\bar\bx,D}+
	\II^{(3)}_{\bar\bx,D}) 
\end{align*}
where for
	$\bar\bx = \{\bx_1,\ldots, \bx_d\} \subset \xi'$
	and  $i=1,2,3$, we
	define $S^{(i)}_{\bar\bx,D} := S^{(i)}_{\bar\bx,D}(\xi'_t)$
	by
\begin{align*}
	\II^{(1)}_{\bar\bx,D} (\xi'_t)
	&:= \sum_{\bar\by 
	\subset 
	\xi'_t \setminus \bar\bx} h^{(r_t)}_D(\bar\by) 
	\1\{ |(\xi'_t \setminus \bar\bx )
	\cap (B(\bq_{\bar\bx})) | < k;
	|(\xi'_t \setminus \bar\by )
	\cap (B(\bq_{\bar\by})) | < k\}
	\\
	& ~~~~~~~~~~~
\times	
	\1\{\|q_{\bar\by} - q_{\bar\bx} \|  > 3 t^\xxi r_t  
	\}; 
	\nonumber \\
	\II^{(2)}_{\bar\bx,D} (\xi'_t)
	&:= \sum_{\bar\by
	\subset 
	\xi'_t \setminus \bar\bx} h^{(r_t)}_D(\bar\by) \1\{ 
	|(\xi'_t \setminus \bar\bx ) \cap (B(\bq_{\bar\bx}))| < k;
	|(\xi'_t \setminus \bar\by ) \cap (B(\bq_{\bar\by}))| < k\}
	\\
	& ~~~~~~~~~~~
 \times	\1\{\|q_{\bar\by} - q_{\bar\bx} \|  \leq 3 t^\xxi r_t \}; 
	\nonumber \\
	\II^{(3)}_{\bar\bx,D}(\xi'_t) 
	&:= \sum_{\emptyset\neq \bar\bz \subsetneq \bx} 
	\sum_{ \bar\by 
	\subset \xi'_t\setminus
	\bar\bx } h^{(r_t)}_D(\bar\bz \cup \bar\by)   \\
	& \hspace{2cm} \times \1\{ |(\xi'_t \setminus (\bar\bz \cup \bar\by))
	\cap B(\bq_{\bar\bz \cup \bar\by}) |
		< k ; | (\xi'_t \setminus \bar\bx) \cap B( q_{\bar\bx})
		| < k \}. 
\end{align*}
\end{lemma}

\begin{proof}
	By definition $F_{t}(D)(F_{t}(D)-1)$ is the number
	of  ordered pairs of distinct subsets $\bar\bx,\bar\by$ of $\xi'_t$
	such that $h_D(\bar\bx) = h_D(\bar\by) =1$
	and $|(\xi'_t \setminus \bar\bx) \cap B(q_{\bar\bx})| <k$
	and $|(\xi'_t \setminus \bar\by) \cap B(q_{\bar\by})| <k$
	with $q_{\bar\bx}$ and $q_{\bar\by}$ both lying in $D$.

	We decompose into the case where
	the sets $\bar\bx, \bar\by$ are disjoint
	and the case where they overlap.
	We further
	decompose the first case according to whether
	or not $q_{\bar\bx}$ and $q_{\bar\by}$ 
	lie `far apart'
	(i.e., further than $3 t^\xxi r_t$).
\end{proof}

It remains to estimate the expected value of the three 
sums arising from Lemma \ref{l:simple}.
For use now and later, given
	 $q,q' \in \RR^d$ we define 
	 the vacancy probabilities
	 \begin{align}
		 w_t := \Pr[
			 |\xi'_t \cap B(q)| < k]
		 ;
		 ~~~~~
		 w^{(2)}_t(q,q') := \Pr[
		\xi'_t(B(q)) < k,  
		\xi'_t(B(q)) < k].
	\label{e:wtdef}
	 \end{align}
	 Note that the first probability above does not depend on $q$.
Also let $\QQ'_t$ be the restriction of $\QQ$ to
$[0,t^\xxi]$, i.e.  $\QQ'_t(da) := {\bf 1}_{[0,t^\xxi]}(a)
	\QQ(da)$. 

\begin{lemma}
	\label{p:I1}
	As $t \to \infty$,
	it holds that
	\begin{align}
		\EE \Big[	\sum_{i \in \cI(t)}
		\sum_{\bar\bx 
		\subset \xi'_t}
		h_{Q_{t,i}^-}^{(r_t)} (\bar\bx)
		\II^{(1)}_{\bar\bx,Q_{t,i}^-}(\xi'_t) \Big] =
		O( r_t^{d/2}). 
		\label{e:I1}
	\end{align}
\end{lemma}
\begin{proof}
	By the multivariate Mecke equation and then\eqref{e:r_t}, 
\begin{align*}
	\EE \Big[\sum_{
		\bar\bx \subset
		\xi'_t}
		h_{Q_{t,i}^-}^{(r_t)}(\bar\bx)
	\II^{(1)}_{\bar\bx,Q_{t,i}^-} \Big]  \le 
	t^{2d}
	\int_{\XX^{2d}} &
	w_t^{(2)}(q_{\bar\bx} ,q_{\bar\by})
	h^{(r_t)}_{Q_{t,i}^-}(\bar\bx)
	h^{(r_t)}_{Q_{t,i}^-}(\bar\by)
	{\bf 1} \{ \|q_{\bar\bx} - q_{\bar\by} \| > 3 t^\xxi r_t \}
	\\
		& (\la_d \otimes \QQ'_t)^{2d}(d(\bar\bx,\bar\by)). 
\end{align*}
Whenever $\|q- q'\| > 3 t^\xxi r_t$ we have
	$w_t^{(2)}(q,q') = w_t^2$, and therefore
	\begin{align*}
	\EE \Big[\sum_{
		\bar\bx \subset
		\xi'_t}
		h_{Q_{t,i}^-}^{(r_t)}(\bar\bx)
		\II^{(1)}_{\bar\bx,Q_{t,i}^-} \Big] 
		\le 
		\big( t^{d}
	\int_{\XX^{d}}
	w_t
		h^{(r_t)}_{Q_{t,i}^-}(\bar\bx)
		 (\la_d \otimes \QQ)^{d}(d\bar\bx)
		 \Big)^2
		 = (\EE [ F_{t}(Q_{t,i}^-)])^2.
	\end{align*}
	Set $b_t := \EE [ F_{t}(Q_{t,i}^-)]$, which is the
	same for all $i \in \cI(t)$.
	Then $b_t \#(\cI(t)) \leq \E[F_t(A)] =O(1)$ by
	Proposition \ref{p:meanlim}. Therefore
	\begin{align*}
		\EE \Big[	\sum_{i \in \cI(t)}
		\sum_{\bar\bx 
		\subset \xi'_t}
		h_{Q_{t,i}^-}^{(r_t)} (\bar\bx)
		\II^{(1)}_{\bar\bx,Q_{t,i}^-}(\xi'_t) \Big] 
		\leq b_t^2 \#(\cI(t)) = O( 1/\#(\cI(t))) = O(r_t^{d/2}),
	\end{align*}
		and hence \eqref{e:I1}
\end{proof}
	Next we give upper and lower bounds on
	the measure of $B(y) \setminus B(x)$ for $x,y \in \R^d$.
	\begin{lemma}
		\label{l:lunevol}
		Let $c_1 := (1/2) \EE[ 
		\min( \theta_{d-1}(Y/2)^{d-1}, \theta_d (Y/4)^d)]$
			and $c_2 := (\theta_d/2) \EE[\min (Y,1/2)^d ]$. 
	Then for all large enough $t$ and all
	$x, y \in \R^d$,
	\begin{align}
		(\la_d \otimes \QQ) (  B(y) \setminus B(x)) & \le \theta_{d-1}
		\E[Y^{d-1}]
		\|y-x\| r_t^{d-1};
		\label{e:volUB}
		\\
		(\la_d \otimes \QQ'_t) (  B(y) \setminus B(x)) & \ge 2 c_1 
		\|y-x\| r_t^{d-1}
		~~~~ \mbox{\rm if}~~ \|y - x \| \leq  r_t;
		\label{e:volLB}
		\\
		(\la_d \otimes \QQ'_t) (  B(y) \setminus B(x)) & \ge 2 c_2 
		 r_t^{d}
		~~~~
		~~~~
		~~~~
		~~~~
		\mbox{\rm if}~~ \|y - x \| \geq  r_t.
		\label{e:volLBfar}
	\end{align}
	\end{lemma}
\begin{proof}
	It is straightforward to obtain \eqref{e:volUB} from
	Fubini's theorem. 
	
	For \eqref{e:volLB}, note that by Fubini's theorem
	\begin{align}
		(\lambda_d \otimes \QQ'_t)(B(y) \setminus B(x))
		= \int_{\R_+}
		\lambda_d(B(y,a r_t) \setminus B(x,ar_t)) \QQ'_t(da).
		\label{e:0616a}
	\end{align}
	Let $a \geq 0$.
	If
	$\|y-x\| \leq (a/2)  r_t$ then $\partial B(y, ar_t) \cap
	\partial B(x,ar_t)$ is a $(d-1)$-dimensional sphere of radius
	at least $(a/2)r_t$  and by an application
	of Fubini's theorem $
		\lambda_d(B(y,a r_t) \setminus B(x,ar_t)) 
		\geq \theta_{d-1}(a/2)^{d-1} r_t^{d-1} \|y-x\|$.

		If instead $(a/2)r_t < \|y-x\| \leq r_t$, then
		$B(y,ar_t) \setminus B(x,ar_t)$ contains a ball of
		radius $(a/4)r_t$, so that
		\begin{align*}
			\lambda_d(B(y,ar_t) \setminus B(x,ar_t)) 
			\geq \theta_d (a/4)^d  r_t^{d}
			\geq \theta_d (a/4)^d r_t^{d-1} \|y-x\|,
		\end{align*}
		and then using \eqref{e:0616a} we obtain \eqref{e:volLB}.

			Now suppose $\|y-x \| \geq r_t$.
			Then $\lambda_d(B(y,ar_t) \setminus B(x,ar_t))
			\geq \theta_d \min(a,1/2)^d r_t^d$. Therefore
			using \eqref{e:0616a} we obtain
			\eqref{e:volLBfar}.
\end{proof}

Next we turn to 
$\II_{\bar\bx,D}^{(2)}(\xi'_t)$. We
write $\II^{(2)}_{\bar\bx,D} = 
	\II^{(2a)}_{\bar\bx,D} +  
	\II^{(2b)}_{\bar\bx,D} 
	+ \II^{(2c)}_{\bar\bx,D} 
	+ \II^{(2d)}_{\bar\bx,D} 
	$, where we set 
	\begin{align*}
		\II^{(2a)}_{\bar\bx,D} := 
		\II^{(2a)}_{\bar\bx,D}(\xi'_t) 
&:= \sum_{\bar\by 
		\subset \xi'_t} h^{(r_t)}_D(\bar\by) \1\{
	|(\xi'_t \setminus \bar\bx) \cap  B(\bq_{\bar\bx})|  < k ,
		|(\xi'_t \setminus \bar\by) \cap 
		B(\bq_{\bar\by}) |   < k 
	\} 
		\\
& \quad \quad \quad \quad 
	\times	
		{\bf 1} \{
			\bar\bx \cap B(\bq_{\bar\by}) =   \bar\by\cap 
			B(\bq_{\bar\bx}) = \emptyset \} 
		\1\{
		\|\bq_{\bar\by}-\bq_{\bar\bx}\|\le  r_t 
	\}; \\
\II^{(2b)}_{\bar\bx,D} 
		:= \II^{(2b)}_{\bar\bx,D} (\xi'_t) 
&:= \sum_{\bar\by 
		\subset \xi'_t} h^{(r_t)}_D(\bar\by) \1\{
	|(\xi'_t \setminus \bar\bx) \cap  B(\bq_{\bar\bx})|  < k ,
		|(\xi'_t \setminus \bar\by) \cap B(\bq_{\bar\by}) |   < k \} 
		\\
& \quad \quad \quad \quad 
	\times	
		{\bf 1} \{ \bar\by \cap B(\bq_{\bar\bx}) \neq \emptyset \} 
		\1\{ \|\bq_{\bar\by}-\bq_{\bar\bx}\|\le  r_t 
	\}; \\
\II^{(2c)}_{\bar\bx,D} 
		:= \II^{(2c)}_{\bar\bx,D} (\xi'_t) 
&:= \sum_{\bar\by 
		\subset \xi'_t} h^{(r_t)}_D(\bar\by) \1\{
	|(\xi'_t \setminus \bar\bx) \cap  B(\bq_{\bar\bx})|  < k ,
		|(\xi'_t \setminus \bar\by) \cap B(\bq_{\bar\by}) |   < k \} 
		\\
& \quad \quad \quad \quad 
	\times	
		{\bf 1} \{ \bar\by \cap B(\bq_{\bar\bx}) = \emptyset \} 
		{\bf 1} \{ \bar\bx \cap B(\bq_{\bar\by}) \neq \emptyset \} 
		\1\{ \|\bq_{\bar\by}-\bq_{\bar\bx}\|\le  r_t 
		\}; \\
\II^{(2d)}_{\bar\bx,D} := 
		\II^{(2d)}_{\bar\bx,D}(\xi'_t) 
&:= \sum_{\bar\by 
		\subset \xi'_t}
		h^{(r_t)}_D(\bar\by) \1\{
	|(\xi'_t \setminus \bar\by) \cap B(\bq_{\bar\by}) |  < k ,
		|(\xi'_t \setminus \bar\bx) \cap B(\bq_{\bar\bx}) |
		< k 
	\} 
		\\
& \quad \quad \quad \quad 
\times
		\1\{  r_t < 
		\|\bq_{\bar\bx}-\bq_{\bar\by}\| \leq 3 t^\xxi   r_t 
	\}. 
	\end{align*}

The expected sum of
$\II_{\bar\bx,Q_{t,i}^-}^{(2a)}$
is perhaps the most  delicate one to deal with,
\blue{and before going into the details in the
proof of Lemma \ref{p:case2} below, we first provide some intuition
(it is easiest at first to keep in mind the case where
$\QQ$ is the Dirac measure $\delta_1$).
When computing the probability that a pair
of $d$-tuples $\bar\bx$ and $\bar\by$ both contribute to 
the sum, we get an extra exponential factor with $-t$ times
the $(\lambda_d \otimes \QQ)$-measure
of $B(\bq_{\bar\by})\setminus B(\bq_{\bar\bx})$
in the exponent, due to vacancy in this set.
 This measure is roughly proportional to $r_t^{d-1}
\|q_{\bar\by}-
q_{\bar\bx}\|$}.
\blue{We shall make a change of variables so as
to integrate over $q_{\bar\bx}$ and over 
$z:= q_{\bar\by}-q_{\bar\bx}$, 
along with the polar coordinates
of the points of $\bar\bx$ relative to $q_{\bar\bx}$
and the points of $\bar\by$ relative to $q_{\bar\by}$.

When $\|z\|$ is small the 
 $(\lambda_d \otimes \QQ)$-measure
of $B(\bq_{\bar\by})\setminus B(\bq_{\bar\bx})$
 can be so small  that our extra exponential factor
 does not provide much decay. While the volume of the region of integration of 
$z$ where $\|z\|$ is small is itself small, it turns
out that the above argument is not quite enough on its own
to give the required decay for the expected sum of
 $S_{\bx,Q_{t,i}^-}^{(2a)}$,
and we shall require a further geometric argument that we now describe.

Given $\bq_{\bar\bx}$ and $\bq_{\bar\by}$,
let $L$ denote the hyperplane equidistant from these two points.
The condition $\bar\bx \cap B(q_{\bar\by}) =
\bar\by \cap B(q_{\bar\bx}) = \emptyset$ implies that
 all points of $\bar\bx$ lie on one side of $L$ and 
 all points of $\bar\by$ lie on the other side of $L$.
Suppose $q_{\bar\by}$ is lower than $q_{\bar\bx}$.
By the hyperplane condition in Lemma \ref{l:hgg},
at least one of the points of $\bar\bx$ must lie
below the hyperplane  $L'$ that is parallel to $L$ 
and contains $q_{\bar\bx}$, and hence 
between $L$ and $L'$. Thus when $\|z\|$ is small,
the spherical location of this point relative to $q_{\bar\bx}$
is confined to a small region.
This extra constraint on the integrals provides an
extra factor of $(\log t)^{-1}$ in our estimate
for the  expected sum of 
 $S_{\bx,Q_{t,i}^-}^{(2a)}$.
}

\begin{lemma}\label{p:case2}
	 As $t \to \infty$, it holds that  
$$
	\EE \Big[ \sum_{i \in \cI(t)}
		\sum_{\bar\bx \subset \xi'_t}
		h_{Q_{t,i}^-}^{(r_t)} (\bar\bx)
		\II^{(2a)}_{\bar\bx,Q_{t,i}^-}(\xi'_t) \Big]
	= O((\log t)^{-1}).
	$$
\end{lemma}

\begin{proof}
	Let $D \subset \R^d$ be a Borel set.
	Applying the Mecke equation leads to 
\begin{align}
	& \EE \Big[ \sum_{\bar\bx
\subset	\xi'_t}
	h_{D}^{(r_t)} (\bar\bx)
	\II^{(2a)}_{\bar\bx,D}(\xi'_t) \Big] 
	\nonumber \\
&\le 
	\frac{t^{2d}}{d!^2}
	 \int_{\XX^{2d}}  h^{(r_t)}_D(\bar\bx) 
	h^{(r_t)}_D(\bar\by) 
	{\bf 1}\{\|\bq_{\bar\bx}-\bq_{\bar\by}\|\le  r_t
	;
	\bar\bx \cap B(q_{\bar\by}) = \bar\by \cap B(q_{\bar\bx})= \emptyset
	\}
	\nonumber \\
	&\quad \quad \quad 
	\times \Pr [ | \xi'_t  \cap  B(\bq_{\bar\by})| < k,
	|\xi'_t \cap  B(\bq_{\bar\bx})| < k ]
	(\la_d \otimes \QQ'_t)^{2d}(d(\bar\bx ,\bar\by)).
	\label{e:0619a}
\end{align}
Here, in a mild abuse of notation, $\bar\bx$ represents
	either the vector $(\bx_1,\ldots,\bx_d) \in \XX^d$ or
	 the corresponding point set $\{\bx_1,\ldots,\bx_d\} \subset \XX^d$,
	according to context,
	and likewise for $\bar\by$.

	The preceding integral is restricted to $(\bar\bx,\bar\by)$
	with $\bar\bx \cap B(q_{\bar\by}) = \bar\by \cap B(q_{\bar\bx}) =
	\emptyset$. For such $(\bar\bx,\bar\by)$,
	we can write
	\begin{align*}
		& \Pr [ |(\bar\bx \cup \xi'_t)  \cap  B(\bq_{\bar\by})| < k,
	|(\bar\by \cup \xi'_t) \cap  B(\bq_{\bar\bx})| < k ]
		\\
		& = \Pr [ | \xi'_t  \cap  B(\bq_{\bar\by})| < k,
	|\xi'_t \cap  B(\bq_{\bar\bx})| < k ]
		\\
		& \leq \Pr [ | \xi'_t  \cap  B(\bq_{\bar\bx})| < k]
		\Pr[ |\xi'_t \cap  B(\bq_{\bar\by})
			\setminus B(\bq_{\bar\bx})| < k ].
	\end{align*}
	Using
	\eqref{e:volUB} and \eqref{e:volLB} we obtain that
	for $q, q' \in \R^d$,
	\begin{align}
		\Pr[ |\xi'_t \cap  
		B(q')
			\setminus B(q)| < k ]
		& \leq \sum_{j=0}^{k-1}
		\frac{(t \theta_{d-1} r_t^{d-1} \EE[Y^{d-1} ]
		\|q'- q\|
		)^j}{j!} \exp(- 2 t c_1 r_t^{d-1} 
			\|q' -q\| 
			)
			\nonumber
			\\
			& \leq c_3 \exp(- c_1 t r_t^{d-1} 
			\|q' - q\|),
		~~~~~{\rm if}~ 
		\|q' - q\|
		\leq r_t,
			\label{e:0612a}
	\end{align}
	for some constant $c_3$ independent of $t$, $q$ and $q'$.
	Also, by \eqref{e:r_t} and 
	the calculation at \eqref{e:fromHolder},
	there is a   constant $c_4$  such that
	for all large enough $t$
	and  any $q \in \RR^d$ we have that
	\begin{align}
		\Pr[| \xi'_t (B(q)) | < k ] & \leq k   (t\theta_d \E[Y^d] 
		r_t^d)^{k-1}
		\exp(-tr_t^d \EE [ Y^d {\bf 1} \{ Y \leq t^\xxi\}]) 
		\nonumber \\
		& \leq  c_4  (\log t)^{k-1}   \Big( \frac{1}{t (\log t)^{d+k-2}
		}
		\Big) 
		=  \frac{c_4}{t (\log t)^{d-1}}.
		\label{e:fromvolLB}
	\end{align}

	For $\bar\bx = (\bx_1,\ldots,\bx_d) 
	\in \XX^d$
	and  $\bar\by = (\by_1,\ldots,\by_d) \in \XX^d$,
	and $i \in [d]$,
	 write $\bx_i = (x_i,a_i)$ and $\by_i = (y_i,b_i)$ with
	$x_i, y_i \in \R^d$ and
	$a_i, b_i \in \R_+$.
	Write $x_i= \bq_{\bar\bx} + r_t a_i u_i$ and $y_i=\bq_{\bar\by} + r_t
	b_i u_{d+i}$, $i\in[d]$, where
$$
u_i:= \Big(\cos \phi_{i1}, \cos \phi_{i2} \sin\phi_{i1}, \ldots , \cos \phi_{i(d-1)} \prod_{j=1}^{d-2}\sin \phi_{ij}, \prod_{j=1}^{d-1}\sin \phi_{ij}\Big)
$$
with $\phi_{ij}\in[0,\pi]$ for $j\in[d-2]$ and $\phi_{i(d-1)}\in [0,2\pi)$, $i\in[2d]$.
	For fixed values of the marks $a_1,\ldots,a_d$ (for $x_1, \ldots,x_d$
	respectively) and $b_1,\ldots,b_d$ (for $y_1,\ldots,y_d$
	respectively),
let $J, J'$ denote the Jacobians in the changes of variables 
\begin{align*}
	\la_d^d(d(x_1,\ldots,x_d)) &= J(\bq_{\bar\bx}, \phi_{[d],.})\la_d(d\bq_{\bar\bx}) \prod_{i\in [d], j\in[d-1]} d\phi_{ij}; \\
	\la^d_d(d(y_1,\ldots,y_d))
	&= J'(\bq_{\bar\by}, \phi_{[2d]\setminus [d],.})\la_d(d\bq_{\bar\by}) \prod_{i\in [2d]\setminus [d], j\in[d-1]} d\phi_{ij}.
\end{align*}
Here, for $I \subset [2d]$,
we use $\phi_{I,.}$ to denote the set of variables $\phi_{ij}, i\in I, 
j\in [d-1]$.
We see that
$J$ is bounded by $c_5 r_t^{d(d-1)} \prod_{i=1}^d a_i^{d-1}$
and $J'$ is bounded by $c_5 r_t^{d(d-1)} \prod_{i=1}^d b_i^{d-1}$
for some universal constant $c_5<\infty$. 
Moreover, when we integrate out each factor of $b_i^{d-1}$ with
respect to the measure $\QQ$ we get a value of  
$\EE[Y^{d-1}]$, which is finite since we assume $\EE[Y^{2(d-1)}]< \infty$.

Let $u$ be the unit directional vector of $\bq_{\bar\bx}-\bq_{\bar\by}$.  We 
claim that if $h^{(r_t)}(\bar\bx) = h^{(r_t)}(\bar\by)=1$,
and 
$\bx_i \notin B(q_{\bar\by})$ and $\by_i \notin B(q_{\bar\bx})$ 
for all $i \in [d]$,
then there exists $i\in[2d]$ such that  
\begin{align}\label{e:u_i}
	|\langle r_t a_i u_i, u \rangle| \le 
	\|\bq_{\bar\bx}-\bq_{\bar\by}\|/2 ,
\end{align}
where for $i \in[d] $ we set $a_{i+d}:= b_i$.
To prove this, suppose that $\bq_{\bar\by}$ is lower than $\bq_{\bar\bx}$. Let $L$ denote the hyperplane such that $\bq_{\bar\by}$ and $\bq_{\bar\bx}$ are mirror to each other  with respect to $L$. 
	For $i \in [d]$, since $x \notin B(q_{\bar\by})$
	we have $\|x_i - \bq_{\bar\bx}\| = a_i r_t <
	\|x_i - \bq_{\bar\by}\|$, so that $x_1,\ldots,x_d$
	all lie 
	above $L$. Let $L^+$ (resp. $L^-$) denote the hyperplane parallel to $L$ that contains $\bq_{\bar\bx}$ (resp. $\bq_{\bar\by}$). 
	From $h^{(r_t)}(\{\bx_1,\ldots,\bx_d\})=1$, using
	Lemma \ref{l:hgg} we see
	that at least one of $x_1,..., x_d$, say $x_1$, lies below
	$L^+$ (but above $L$). As a consequence, the projection of
	$x_1-\bq_{\bar\bx}$ to $u$ is bounded by $\|\bq_{\bar\bx} - \bq_{\bar\by}\|/2$,
	as required. The case where $\bq_{\bar\bx}$ is lower than $\bq_{\bar\by}$ is
	completely analogous, with $x_i$ replaced by $y_i$, so that
	\eqref{e:u_i} holds for some $u_i, i\in [2d]\setminus [d]$. 

	By \eqref{e:0619a}, \eqref{e:0612a} \eqref{e:fromvolLB}
	and the preceding discussion,
for any bounded Borel $D \subset \RR^d$,
\begin{align*}
	\EE \Big[\sum_{\bar\bx \subset \xi'_t}
	h^{r_t}_{D}(\bar\bx)
	\II^{(2a)}_{\bar\bx,D} (\xi'_t)
	\Big]
	\le & \frac{ c_5^2 c_4 (\EE[Y^{d-1}])^{2d-1} 
	t^{2d}  r_t^{2d(d-1)}
	}{d! t (\log t)^{d-1}}
	\sum_{i=1}^{2d}
	 \\ &
	 \times \int \cdots \int a^{d-1}
	 e^{-c_1 t \|\bq_{\bar\by}-\bq_{\bar\bx}\| r_t^{d-1}}
	\1\{ \|\bq_{\bar\by}-\bq_{\bar\bx}\|\le r_t, \bq_{\bar\bx}\in D
	\} \\ & 
	 \times \1\{ |\langle r_t a u_i, u\rangle| \le \|\bq_{\bar\by}
	-\bq_{\bar\bx}\|  \} 
	\QQ(da) \la^{2}_d(d(\bq_{\bar\bx}, \bq_{\bar\by})) \prod_{i\in[2d],j\in[d-1]} d\phi_{ij}.
\end{align*}

Let $v$ denote the surface measure on the unit sphere $\mathbb{S}_{d-1}$. 
	There is a finite constant  $c_6$ such that
	$v(\{u_1\in \mathbb{S}_{d-1}: |\langle a r_t
	u_1 ,u\rangle| \le \|\bq_{\bar\by}-\bq_{\bar\bx}\|\})\le c_6
	\|\bq_{\bar\by}-\bq_{\bar\bx}\|/ (a r_t)$ for all $t$, all $a>0$
	and all $\bq_{\bar\bx}, \bq_{\bar\by}$
	with $\|\bq_{\bar\by}-\bq_{\bar\bx}\| \leq r_t$. Hence, for some finite $c_7$,
\begin{align*}
	\EE \Big[ \sum_{\bar\bx \subset
	\xi'_t}
	h_{D}^{(r_t)}(\bar\bx)
	\II^{(2a)}_{\bar\bx,D} \Big]
	&
	\le  c_7  t^{2d-1} (\log t)^{1-d} r_t^{2d(d-1)}
	\int
	e^{-c_1 t \|\bq_{\bar\by}-\bq_{\bar\bx}\|r_t^{d-1}  }
	(\|\bq_{\bar\by}-\bq_{\bar\bx}\|/r_t) \\
	&
	\hspace{2cm}
	\times \1\{\|\bq_{\bar\by}-\bq_{\bar\bx}\|\le r_t,
	\bq_{\bar\bx}\in D\}  \la^{2}_d(d(\bq_{\bar\bx} ,\bq_{\bar\by})).
\end{align*}
By changing  variables from
$\bq_{\bar\by}$ to
 $z = \bq_{\bar\by} - \bq_{\bar\bx}$ and then to
 $w = t r_t^{d-1} z$,
we see that the second factor inside the last integral becomes
$\|w\|/(tr_t^d)$ and  the
last integral as a whole is $\lambda_d(D) \times
O((tr_t^d)^{-1} (tr_t^{d-1})^{-d}),$ uniformly in $d$.
Hence
\begin{align}
	& \EE \Big[ \sum_{i \in \cI(t)} 
	\sum_{\bar\bx \subset \xi'_t}
	h^{r_t}_{Q^-_i} (\bx) \II^{(2a)}_{\bar\bx,Q_{t,i}^-} \Big]
	= O(t^{d-2} (\log t)^{1-d}  r_t^{d(d-2)} )
	= O((\log t)^{-1}),
	\label{e:I2aest}
\end{align}
as required.
\end{proof}

Next we consider 
$\II_{\bar\bx,Q_{t,i}^-}^{(2b)}$
and $\II_{\bar\bx,Q_{t,i}^-}^{(2c)}$. Note that these sums are zero in the
case where $k=1$, so the reader interested only in \blue{that} case can 
skip the next
lemma. 
\blue{
For general $k$, in
$\II_{\bar\bx,D}^{(2b)}$
the condition $\bar\by \cap B(q_{\bar\bx}) \neq \emptyset$ 
means that the number of variables of integration when
we apply the Mecke formula is reduced by at least 1,
which translates into a reduction of the integral 
by a factor of $O(tr_t^d)$ and hence by a factor of $O((\log t)^{-1})$.
}
\begin{lemma}\label{p:case2bc}
	 As $t \to \infty$, it holds that  
$$
	\EE \Big[ \sum_{i \in \cI(t)}
		\sum_{\bar\bx \subset \xi'_t}
		h_{Q_{t,i}^-}^{(r_t)} (\bar\bx)
		(\II^{(2b)}_{\bar\bx,Q_{t,i}^-}(\xi'_t) +
		\II^{(2c)}_{\bar\bx,Q_{t,i}^-}(\xi'_t))
		\Big]
	= O((\log t)^{-1}).
	$$
\end{lemma}
\begin{proof}
By the Mecke formula, for any bounded Borel $D \subset \R^d$,
\begin{align*}
	& \EE \Big[ \sum_{\bar\bx 
	\subset \xi'_t}
	h_{D}^{(r_t)} (\bar\bx)
	\II^{(2b)}_{\bar\bx,D}(\xi'_t) \Big] 
	\\
&\le 
	\frac{t^{2d}}{d!^2 } 
	 \int_{\XX^{2d}}  h^{(r_t)}_D(\bar\bx) 
	h^{(r_t)}_D(\bar\by) 
	{\bf 1}\{\|\bq_{\bar\by}-\bq_{\bar\bx}\|\le 2 r_t
	; \bar\by \cap B(q_{\bar\bx}) \neq \emptyset
	\}
	\\
	&\quad \quad \quad 
	\Pr [ |(\bar\bx \cup \xi'_t)  \cap  B(\bq_{\bar\by})| < k,
	|(\bar\by \cup \xi'_t) \cap  B(\bq_{\bar\bx})| < k ]
	(\la_d \otimes \QQ'_t)^{2d}(d(\bar\bx ,\bar\by)).
\end{align*}
Assume without loss of generality that $k \geq 2$.
When
	$\|\bq_{\bar\by}-\bq_{\bar\bx}\|\le  r_t $ and
	$\bar\by \cap B(q_{\bar\bx}) \neq \emptyset $, 
		using \eqref{e:0612a} and \eqref{e:fromvolLB} 
	we have
		that
	\begin{align*}
		& \Pr [ |(\bar\bx \cup \xi'_t)  \cap  B(\bq_{\bar\by})| < k,
	|(\bar\by \cup \xi'_t) \cap  B(\bq_{\bar\bx})| < k ]
		\\
		& \leq \Pr[ |\xi'_t \cap B( q_{\bar\bx}) |
		< k-1]
		\times
		 \Pr[ |\xi'_t \cap B(q_{\bar\by})
		 \setminus B(q_{\bar\bx})| < k]
		 \\
		 & \leq \frac{c_6 c_4  (tr_t^d)^{k-2}}{t (\log t)^{d+k-2}}
		 \exp(-c_1 t r_t^{d-1}
		 \|\bq_{\bar\by}
		 - \bq_{\bar\bx}\| ) .
	\end{align*}
	Therefore by going into polar coordinates as before
	but this time without any geometrical constraint, we
	obtain for some constant $c$ that
	\begin{align*}
	& \EE \Big[ \sum_{\bar\bx
		\subset \xi'_t}
		h_{D}^{(r_t)} (\bar\bx)
		\II^{(2b)}_{\bar\bx,D}(\xi'_t) \Big] 
		\leq \frac{ ct^{2d-1}
		(tr_t^d)^{k-2} r_t^{2d(d-1)} }{
			(\log t)^{d+k-2}}
	 \\ &
	 \times \int \cdots \int e^{-c_1 t
		r_t^{d-1}
		\|\bq_{\bar\by}-\bq_{\bar\bx}\| 
		}
		\1\{ 
		\bq_{\bar\bx}\in
		D
	\}
\la^{2}_d(d(\bq_{\bar\bx}, \bq_{\bar\by})) \prod_{i\in[2d],j\in[d-1]} d\phi_{ij}.
	\end{align*}
	Changing variables from $q_{\bar\by}$ to
	$z= q_{\bar\by}-q_{\bar\bx}$
	and then to $w= tr_t^{d-1}z$, we obtain that
	\begin{align*}
		& \EE \Big[ \sum_{i \in \cI(t)}
		\sum_{\bar\bx 
		\subset \xi'_t}
		h_{Q_{t,i}^-}^{(r_t)} (\bar\bx)
		\II^{(2b)}_{\bar\bx,Q_{t,i}^-} \Big] 
		\leq \frac{c t^{2d+k-3} r_t^{d(2d+k-4)}}{(\log t)^{d+k-2}}
		\sum_{i \in \cI(t)}
		 \lambda_d(Q_{t,i}^-)
		\int e^{-c_1 \|w\|} (t r_t^{d-1})^{-d} dw,
	\end{align*}
	which is $O((\log t)^{-1})$. The sum involving
	 $\II^{(2c)}(\bar\bx)$
	is treated in the same way with the roles of $\bar\bx$ and $\bar\by$
	reversed.
	\end{proof}

Next we deal with  $\II_{\bar\bx,Q_{t,i}^-}^{(2d)}$. 
\blue{In this case $q_{\bar\bx}$ and $q_{\bar\by}$ are well-separated
so the `extra exponential factor' described in the discussion
prior to Lemma \ref{p:case2}
 provides an effective bound.} 
\begin{lemma}\label{p:case2d}
	There exits $\delta >0$ such that
	 as $t \to \infty$, it holds that  
	\begin{align}
	\EE \Big[ \sum_{i \in \cI(t)}
		\sum_{\bar\bx \subset \xi'_t}
		h_{Q_{t,i}^-}^{(r_t)} (\bar\bx)
		\II^{(2d)}_{\bar\bx,Q_{t,i}^-}(\xi'_t) 
		\Big]
	= O(t^{-\delta}).
	\label{e:0619b}
\end{align}
\end{lemma}
\begin{proof}
	Set $c_2 := (\theta_d/2) \EE[\min (Y,1/2)^d ]$ as before. 
	By \eqref{e:volLBfar}, 
	for all large enough $t$ and all $x,y \in \R^d$
we have  that
\begin{align}
	\PP[ | \xi'_t \cap B(y) \setminus B(x) | <k]
	\leq e^{-c_2 t r_t^d} ~~
	\mbox{whenever} ~~ \|y - x\| \geq  r_t. 
	\label{e:LBfar}
\end{align}
	Hence, with $w_t$ defined at \eqref{e:wtdef},
	  if $\|q_{\bar\by}- q_{\bar\bx}\| > r_t$, we have
\begin{align*}
	\Pr [ |(\bar\bx \cup \xi'_t)  \cap  B(\bq_{\bar\by})| < k,
	|(\bar\by \cup \xi'_t) \cap  B(\bq_{\bar\bx})| < k ]
	\leq w_t e^{-c_2 tr_t^d}.
\end{align*}
We apply the Mecke equation and
	\eqref{e:r_t} to deduce that for any $i \in \cI(t)$,
\begin{align*}
	& \EE
	\Big[ 
	\sum_{\bar\bx \subset \xi'_t}
	h^{(r_t)}_{Q_{t,i}^-}(\bar\bx)
	\II^{(2d)}_{\bar\bx,Q_{t,i}^-} (\xi'_t) \Big] 
	\\
	&\le 
	\frac{t^{2d}}{d!^2}
	\int  h^{(r_t)}_{Q_{t,i}^-}(\bar\bx) h^{(r_t)}_{Q_{t,i}^-}
	(\bar\by) w_t e^{- c_2 t r_t^{d} }
	{\bf 1}\{\|q_{\bar\by} - q_{\bar\bx}\| \in (r_t,3t^\xxi r_t]\}
	(\la_d \otimes \QQ'_t)^{2d}(d(\bar\bx,\bar\by)).
\end{align*}
Going into polar coordinates as before for $\bar\by$ but not $\bar\bx$
gives us
a factor of $r_t^{d(d-1)}$. Writing $z= \bq_{\bar\by}-\bq_{\bar\bx}$,
we obtain that the last expression is
bounded by a constant times
\begin{align*}
	t^{2d} w_t e^{-c_2tr_t^d} \Big( \int h^{(r_t)}_{Q^-_{t,i}}(\bar\bx)
	(\lambda_d \otimes \QQ'_t)^d(d\bar\bx) \Big)
	\times r_t^{d(d-1)} (t^\xxi r_t)^d.
	\\
	= 
	\Big(
	t^d w_t\int h^{(r_t)}_{Q_{t,i}^-}(\bar\bx)
	(\lambda_d \otimes \QQ'_t)^d(d\bar\bx) \Big)
	\times
	t^{d \xxi}  e^{-c_2tr_t^d} 
	(t r_t^{d})^d.
\end{align*}
The first factor (in large  brackets), summed over
	$i \in \cI(t)$, equals
	$\EE [F_t(\cup_{i \in \cI(t)} Q_{t,i}^-)]$, and
	is therefore  $O(1)$,
	while by \eqref{e:r_t} we
	have $e^{-c_2 tr_t^d} \leq t^{-  c_2/(2 \alpha)}$ for
	$t$ large,
	and since $\xxi d < c_2/(4 \alpha)$ by \eqref{e:xxi}
 the second factor is 
 $O(t^{-c_2/(4 \alpha)} (\log t)^d)$.
Hence for some $\delta >0$, we have \eqref{e:0619b}.
\end{proof}

\blue{
The sum
 $\II_{\bar\bx,Q_{t,i}^-}^{(3)}$
 is handled in Lemma \ref{p:case3} below. 
The intuition is similar to that described for
 $\II_{\bar\bx,Q_{t,i}^-}^{(2a)}$
just before Lemma \ref{p:case2}, except that the final geometrical
argument described there no longer applies. What we have here instead,
since we are now considering pairs of non-disjoint $d$-tuples,
is that when we apply the Mecke formula here we are integrating
over
 $(d-j)$-tuples
 $\bar\by$ for some $j \in [1,k-1]$, rather than $d$-tuples as
the Lemma \ref{p:case2}. Having $j$ fewer variables of integration
causes us to lose a factor of $(tr_t^d)^j$, 
which decreases the expectation by a factor of $(\log t)^j$.
  }

\begin{lemma}
	\label{p:case3}
As $t \to \infty$ it holds that 
	\begin{align}
		\EE \Big[ \sum_{i \in \cI(t)}
		\sum_{\bar\bx \subset \xi'_t}
		\II^{(3)}_{\bar\bx,Q_{t,i}^-}(\xi'_t)
		\Big]
		= O((\log t)^{-1}).
		\end{align}
\end{lemma}

\begin{proof}
	Given bounded Borel $D \subset \R^d$, and $x \in \xi'_t$,
	for each $j \in [d-1]$,
	let $\II_{\bar\bx,D}^{(3,j)}$ be
	the contribution of sets $\bar\bz \subset \bar\bx$ of
	cardinality $j$ to $\II_{\bar\bx,D}^{(3)}(\xi'_t)$.
	Also define
	\begin{align*}
		T^{(j)}_{\bar\bx,D} &:= 
		\sum_{\bar\bz \subset \bar\bx:|\bar\bz|=j}
		\sum_{ \bar\by 
		\subset\xi'_t\setminus
		\bar\bx } h^{(r_t)}_D(\bar\bz \cup \bar\by)   
		{\bf 1}\{\|\bq_{\bar\bz \cup \bar\by} -
		\bq_{\bar\bx}\|\leq r_t\}
		\\
		& \hspace{2cm} \times
		\1\{ |
		(\xi'_t \setminus (\bar\bz \cup \bar\by)) \cap 
		B(\bq_{\bar\bz \cup \bar\by}) | < k \} 
		\1\{ | (\xi'_t \setminus \bx) \cap
		B(\bq_{\overline{\bx}}) | < k  \};
\\
		U^{(j)}_{\bar\bx,D} &:= 
		\sum_{\bar\bz \subset \bar\bx:|\bar\bz| =j}
		\sum_{ \bar\by 
		\subset 
		\xi'_t\setminus \bar\bx} 
		h^{(r_t)}_D(\bar\bz \cup  \bar\by)   
		{\bf 1}\{\|\bq_{\bar\bz \cup \bar\by} - \bq_{\bar\bx}\|>  r_t\}
		\\
		& \hspace{2cm} \times \1\{ |
		(\xi'_t \setminus (\bar\bz \cup \bar\by)) \cap 
		B(\bq_{\bar\bz \cup \bar\by}) | <
		k \} 
		\1\{ | (\xi'_t \setminus \bar\bx) \cap
		B(\bq_{\bar\bx}) | < k  \}.
\end{align*}
	Then $\II^{(3)}_{\bar\bx,D}(\xi'_t) =
	\sum_{j=1}^{d-1} \II_{{\bar\bx,D}}^{(3,j)} 
	= \sum_{j=1}^{d-1}( T_{{\bar\bx,D}}^{(j)} + U_{{\bar\bx,D}}^{(j)})$.

	By the Mecke equation,
	 \eqref{e:fromvolLB},
	 \eqref{e:r_t} and  \eqref{e:0612a}, there exists $c_8 >0$ such that
\begin{align}
&
	 \EE \Big[ \sum_{\bar\bx \subset \xi'_t}
	 T^{(j)}_{\bar\bx,D} h_D^{(r_t)} (\bar\bx) \Big]
	 = \frac{t^{2d-j}}{j!(d-j)!^2}
	 \int_{\XX^{d}} 
	 \int_{\XX^{d-j}}
	 h_D^{(r_t)} (\bar\bx) 
	 h_D^{(r_t)} (\bar\bx_{[j]} \cup \bar\by) 
	 \nonumber
	 \\
	 & \times 
	 \Pr[ |(\xi'_t \cup \bar\by) \cap B(q_{\bar\bx})| <k,
	  |(\xi'_t \cup  \bar\bx_{[j]} \cup \bar\by) \cap 
	  B(q_{\bar\bx_{[j]} \cup \bar\by})| <k ] 
	(\lambda_d \otimes \QQ'_t)^{d-j}(d\bar\by)
	(\lambda_d \otimes \QQ'_t)^{d}(d\bar\bx)
	\nonumber \\
	&	\le \frac{c_8  t^{2d-j}
	}{t(\log t)^{d-1} 
	}
	\int_{\XX^d} \int_{\XX^{d-j}}
	h^{(r_t)}_D(\bar\bx) h^{(r_t)}_D(\bar\bx_{[j]} \cup \bar\by)
	\times \exp(- c_1 t \|\bq_{\bar\bx} - \bq_{\bar\bx_{[j]} 
	\cup \bar\by}\| r_t^{d-1} )
	\nonumber
	\\
	&
	\quad \quad \quad
	\quad \quad \quad
	\quad \quad \quad
	\times
	(\la_d \otimes \QQ'_t)^{d-j} (d\bar\by)
	(\la_d \otimes \QQ'_t)^{d} (d\bar\bx),
	\label{e:IIa}
\end{align}
	where here $\bar\by := (\by_1,\ldots,\by_{d-j})$,
	identified with $\{\by_1,\ldots,\by_{d-j}\}$ according to context,
	and $\bar\bx_{[j]} :=
	(\bx_1,\ldots,\bx_j)$, identified with
	$\{\bx_1,\ldots,\bx_{j}\}$ according to context.
	Below we shall  write 
	$q_{\overline{\bx_{[j]} \by}}$
	for 
	$q_{\bar{\bx}_{[j]} \cup \bar\by}$.

	For each $i \in [d]$ write $\bx_i = (x_i,a_i)$,
	and if $i \leq d-j$ write $\by_i = (y_i,b_i)$,
	with $x_i, y_i \in \R^d$ and $a_i,b_i \in \R_+$.
	Given the values of $\bq_{\bar\bx}$ and $\bq_{\overline{\bx_{[j]}\by}}$,
	by definition for each $i \in [j]$ we have $\|x_i - q_{\bar\bx}\|
	= \|x_i - q_{\overline{\bx_{[j]}\by}} \| = r_t a_i$.
	In particular
	$\|\bq_{{\bar\bx}} - \bq_{\overline{\bx_{[j]}\by}}\|\leq 2 r_t
	\min_{i \in [j]} a_i$.
	Set $z:= \bq_{\overline{\bx_{[j]}\by}} - \bq_{\bar\bx}$.
	For each $i \in [j]$
	the vector $x_i$ lies on the surface of the sphere
	centred on
	$\bq_{\bar\bx} + \frac12 z$
	of radius
	$\sqrt{(a_i r_t)^2 - (\|z\|/2)^2} $
	and also on the hyperplane through 
	$\bq_{\bar\bx} + \frac12 z $
	perpendicular to $z$.
	 Hence for each $i \in [j]$ we can write
	 \begin{align*}
		 x_i = 
		 \bq_{\bar\bx} + (1/2) z
		 + ((a_i r_t)^2 - (\|z\|/2)^2)^{1/2}
		 \Theta_{ \|z\|^{-1} z }((\omega_i,0))
	 \end{align*}
	 for some $\omega_i \in \mathbb{S}_{d-2}$,
	 where
	 for any unit vector $e$ in $\R^d$ we denote by $\Theta_e$
	 the geodesic rotation that sends $\bfe_d$ to $e$, and
	 $\mathbb{S}_{k}$ denotes the  $k$-dimensional
	 unit sphere for all $k \in \NN$.

	Write 
  $x_i= \bq_{\bar\bx}  + r_t a_i u_i$ for $i\in [d]\setminus [j]$ and
	$y_i= \bq_{\overline{\bx_{[j]} \by}} + r_t b_i u'_i, i\in [d-j]$.
	Write $v_k$ for the $k$-dimensional surface measure
	on 
	$\mathbb{S}_k$;  
as in the proof of Proposition \ref{p:case2}, the $k$-dimensional surface
integral can be written in terms of $k$ real-valued polar
coordinates.
There is a finite constant $c_9$ such that
	for any
	fixed $(a_1,\ldots,a_d, b_1,\ldots,b_{d-j})$ the Jacobian $J$ in 
 \begin{align*}
	 \la_d^{2d-j}(d(x_1,\ldots,x_d,y_1,\ldots,y_{d-j}) )=& 
	 J(\bq_{\bar\bx}, \bq_{\overline{\bx_{[j]}\by}},\omega_{[j]}, 
	 u_{[d]\setminus [j]},
	 u'_{[d-j]}) 
	 v_{d-2}^j
	 (d (\omega_1,\ldots,\omega_j))
	 \\
	 & \times 
	 v_{d-1}^{d-j}(d(u_{j+1},\ldots, u_d))
	 v_{d-1}^{d-j}(d(u'_1,\ldots,u'_{d-j})) 
	 dz
	 d\bq_{{\bar\bx}}
\end{align*}
satisfies 
\begin{align*}
	J & \le  c_9 r_t^{j (d-2) + 2(d-j)(d-1)}
	\prod_{i\in [j]} a_i^{d-2} \prod_{i \in [d] \setminus [j]}
a_i^{d-1} \prod_{i \in [d-j]} b_i^{d-1}
	\\ & 
=  
c_9
r_t^{2d^2 -d(2+j)}
	\prod_{i\in [j]} a_i^{d-2} \prod_{i \in [d] \setminus [j]}
a_i^{d-1} \prod_{i \in [d-j]} b_i^{d-1}.  
\end{align*}
Therefore using \eqref{e:IIa},
writing $q$ for 
$q_{\bar\bx}$ we obtain that
\begin{align*}
	\EE \Big[ 
	\sum_{\bar\bx \subset \xi'_t}
	h_D^{(r_t)}(\bar\bx)
	T^{(j)}_{\bar\bx,D} 
%
	\Big]
	\le
	& c_4 c_9 
	(\E[Y^{d-2}])^j (\E[Y^{d-1}])^{2d-j}
	\Big( \frac{
	t^{2d-j -1} r_t^{2d^2 - d(2+j)}
	}{(\log t)^{d-1} 
	}
\Big)
	\\ & ~~~~~~
	\times 
	\int_{D} 
	\int_{\RR^d}  
	\exp( -c_1 t r_t^{d-1} \|z\| )
	\lambda_d(dz) \lambda_d(d\bq).
\end{align*}
Changing variable to $w = tr_t^{d-1} z$ gives us a further factor
of $(tr_t^{d-1})^{-d}$. Therefore 
\begin{align}
	 \EE \Big[ \sum_{i \in \cI(t)}
	 \sum_{\bar\bx \subset \xi'_t}
	 h_{Q_{t,i}^-}^{(r_t)}(\bar\bx)
	 T^{(j)}_{\bar\bx,Q_{t,i}^-} \Big]
	= O((\log t)^{1-d} t^{d-j  -1} r_t^{d^2 -d -jd } ) =
O((\log t)^{-j})
	.
	\label{e:IIaest}
\end{align}

Now consider $U_{\bar\bx,D}^{(j)}$. Using
 \eqref{e:LBfar}
 we obtain a similar bound to \eqref{e:IIa} but with 
 $U_{\bar\bx,D}^{(j)}$ replacing 
 $T_{\bar\bx,D}^{(j)}$ on the left and
 a factor 
 $e^{-c_2t r_t^d}$
 replacing  the factor of
	$ \exp(- c_1 t \|\bq_{\bar\bx} - \bq_{\overline{\bx_{[j]} \by}}\|
	r_t^{d-1} ) $
 on the right.

 We make the same change of variables as before, noting that
 for each $i \in [j]$ we require $r_ta_i \geq \|
 \bq_{\bar\bx}
 -
 \bq_{\overline{\bx_{[j]}\by}}
 \|/2$.
 Hence for the integral with respect to $\QQ'_t(da_i) v_{d-2} ( d\omega_i)$
 we obtain a factor  bounded by a constant times
 \begin{align*}
	 \int_{[\|\bq_{\bar\bx} - q_{\overline{\bx_{[j]}\by}}
	 \|/(2r_t),\infty)} a^{d-2} \QQ'_t(da)
	 =  \EE [Y^{d-2} {\bf 1}\{
		 \|\bq_{\bar\bx} - \bq_{\overline{\bx_{[j]}\by}}\|/(2r_t)
		 \leq Y
		 \leq 
		 t^\xxi 
	 \}].
 \end{align*}
Therefore integrating out everything
 except the variables $\bq_{\bar\bx}$ and $\bq_{\overline{\bx_{[j]} \by}}$
 and writing $z$ for $\bq_{\overline{\bx_{[j]}\by}} - \bq_{\bar\bx}$ and
 then $w= (2 r_t)^{-1} z$ we obtain 
 similarly to before that
 there is a finite constant $c_{10}$ such that
\begin{align*}
	\EE \Big[ \sum_{i \in \cI(t)}
	\sum_{\bar\bx 
	\subset \xi'_t} h_{Q_{t,i}^-}^{(r_t)} (\bar\bx)
	U^{(j)}_{\bar\bx,Q_{t,i}^-} \Big] 
	& \le 
	 \frac{c_{10} t^{2d-j -1} r_t^{2d^2 - d(2+j)} 
	 }{
		 (\log t)^{d-1}}
	 \times \sum_{i \in \cI(t)}
	 \\
	 & \times 
	\int_{Q^-_{t,i}} 
	\int_{\R^d}  
	e^{ -c_2 t r_t^{d}  }
	(\E[ Y^{d-2} {\bf 1}\{t^\xxi \geq Y \geq \|z\|/r_t\} ])^j
	\lambda_d(dz) \lambda_d(d\bq)
	\\
	& \leq 
	 \frac{2^d c_{10} 
	 t^{2d-j  -1} r_t^{2d^2 - d(1+j) }
	e^{-c_2tr_t^d}
	}{(\log t)^{d-1}
	}
	\int_{\R^d} \Big( \int_{[\|w\|,\infty)} a^{d-2} \QQ(da) \Big)^j
	dw.
\end{align*}
Since we
assume $\E[Y^{2d-2}] < \infty$ the  integral on the last line is finite.
Indeed, for $j =1$, by Fubini's theorem it is bounded by
$
\theta_d \int_{\R_+} a^{2d-2} \QQ(da), 
$
which is finite. For general $j \geq 1$,
choosing $K \in \R_+$ such that
$\E[Y^{d-2} {\bf 1}\{ Y \geq K\} ] \leq 1$, we find
that the contribution to the integral from $(w,a)$ with
$\|w\| >K $ is no greater for $j >1$ than it would be for $j=1$
and therefore is finite, while the contribution
from $(w,a)$ with $\|w\| \leq K$ is also finite.
Thus 
\begin{align*}
	\EE \Big[ \sum_{i\in \cI(t)}
	\sum_{\bar\bx 
	\subset \xi'_t}
	h_{Q_{t,i}^-}^-(\bar\bx)
	U^{(j)}_{\bar\bx,Q_{t,i}^-} \Big]
	= O\Big(\Big(\frac{(tr_t^d)^{2d-j-1}}{(\log t)^{d-1} }
	\Big) e^{-c_2 tr_t^d}\Big)
	= O(e^{-(c_2 /2) tr_t^d}),
\end{align*}
	and combined with \eqref{e:IIaest} this gives us the result.
\end{proof}

\begin{proof}[Proof of Proposition  \ref{p:sumbd}]
	 Lemmas \ref{l:simple} -- \ref{p:I1}
	and 
	\ref{p:case2} --
	\ref{p:case3}
	yield \eqref{e:Esumbd}, 
	as required.
\end{proof}

{\bf Acknowledgements.}
 This research was supported by EPSRC grant EP/T028653/1.
 For the purpose of open access, the authors have applied a Creative Commons Attribution (CC-BY) licence to any Author Accepted Manuscript version arising.

  \appendix
  
\bibliographystyle{plain}

\end{document}